\pdfoutput=1

\documentclass{amsart}

\usepackage[T1]{fontenc}
\usepackage[utf8]{inputenc}
\usepackage[USenglish]{babel}
\usepackage{lmodern}
\usepackage{amssymb}
\usepackage{enumerate}
\usepackage{mathtools}
\usepackage{tikz-cd}
\usepackage[linktocpage=false]{hyperref}

\DeclareFontFamily{OMX}{lmex}{}
\DeclareFontShape{OMX}{lmex}{m}{n}{<->lmex10}{}

\tikzcdset{arrow style=Latin Modern}

\theoremstyle{plain}
\newtheorem{thm}[equation]{Theorem}
\newtheorem{cor}[equation]{Corollary}
\newtheorem{prop}[equation]{Proposition}
\newtheorem{lem}[equation]{Lemma}

\theoremstyle{definition}
\newtheorem{defn}[equation]{Definition}
\newtheorem{hyp}[equation]{Hypothesis}

\theoremstyle{remark}
\newtheorem{rem}[equation]{Remark}
\newtheorem{exmp}[equation]{Example}

\numberwithin{equation}{section}

\DeclareMathOperator{\id}{id}
\DeclareMathOperator{\Hom}{Hom}
\DeclareMathOperator{\End}{End}
\DeclareMathOperator{\Ext}{Ext}
\DeclareMathOperator{\Ind}{Ind}
\DeclareMathOperator{\Indd}{I}
\DeclareMathOperator{\Ord}{Ord}

\DeclareMathOperator{\Art}{Art}
\DeclareMathOperator{\Noe}{Noe}
\DeclareMathOperator{\Pro}{Pro}
\DeclareMathOperator{\Res}{Res}
\DeclareMathOperator{\Def}{Def}
\DeclareMathOperator{\Mod}{Mod}
\DeclareMathOperator{\Ban}{Ban}
\DeclareMathOperator{\Cat}{Cat}
\DeclareMathOperator{\Set}{Set}
\DeclareMathOperator{\Ob}{Ob}
\DeclareMathOperator{\ord}{ord}
\DeclareMathOperator{\diag}{diag}
\newcommand{\iso}{\overset{\sim}{\longrightarrow}}
\newcommand{\vertiso}{\rotatebox{-90}{\(\sim\)}}
\newcommand{\otimesh}{\mathbin{\widehat{\otimes}}}
\newcommand{\middlevert}{\;\middle|\;}
\newcommand{\N}{\mathbb{N}}
\newcommand{\Z}{\mathbb{Z}}
\newcommand{\Q}{\mathbb{Q}}
\newcommand{\F}{\mathbb{F}}
\newcommand{\OO}{\mathcal{O}}
\newcommand{\CC}{\mathcal{C}}
\newcommand{\ZZ}{\mathcal{Z}}
\newcommand{\NN}{\mathcal{N}}
\newcommand{\m}{\mathfrak{m}}
\newcommand{\af}{\mathfrak{a}}
\newcommand{\gf}{\mathfrak{g}}
\newcommand{\tf}{\mathfrak{t}}
\newcommand{\HOrd}[1][\bullet]{\mathrm{H}^{#1}\!\Ord}
\newcommand{\Deft}{\widetilde{\Def}}

\newcommand{\GL}{\mathrm{GL}}
\newcommand{\proaug}{\mathrm{pro~aug}}
\newcommand{\lfin}{\mathrm{l.fin}}
\newcommand{\fl}{\mathrm{fl}}

\newcommand{\adm}{\mathrm{adm}}
\newcommand{\cont}{\mathrm{cont}}
\newcommand{\unit}{\mathrm{unit}}
\newcommand{\univ}{\mathrm{univ}}

\makeatletter
\def\@@and{\unskip}
\makeatother

\hypersetup{pdfinfo={
	Title={Deformation rings and parabolic induction},
	Author={Julien Hauseux, Tobias Schmidt, Claus Sorensen},
	Subject={22E50},
	Keywords={p-adic reductive groups, smooth representations, m-adically continuous representations, parabolic induction, deformations}
}}

\title{Deformation rings and parabolic induction}

\author[J.~Hauseux]{Julien Hauseux}
\address{
Université de Lille\\
Département de Mathématiques\\
Cité scientifique, Bâtiment M2\\
59655 Villeneuve d'Ascq Cedex\\
France
}
\email{\href{mailto:julien.hauseux@math.univ-lille1.fr}{julien.hauseux@math.univ-lille1.fr}}
\thanks{J.H. was partly supported by EPSRC grant EP/L025302/1.}

\author[T.~Schmidt]{Tobias Schmidt}
\address{
Institut de Recherche Mathématique de Rennes\\
Campus Beaulieu\\
35042 Rennes Cedex\\
France
}
\email{\href{mailto:tobias.schmidt@univ-rennes1.fr}{tobias.schmidt@univ-rennes1.fr}}

\author[C.~Sorensen]{Claus Sorensen}
\address{
Department of Mathematics, UCSD\\
9500 Gilman Dr. \#0112\\
La Jolla, CA 92093-0112\\
USA
}
\email{\href{mailto:csorensen@ucsd.edu}{csorensen@ucsd.edu}}

\begin{document}

\begin{abstract}
We study deformations of smooth mod $p$ representations (and their duals) of a $p$-adic reductive group $G$.
Under some mild genericity condition, we prove that parabolic induction with respect to a parabolic subgroup $P=LN$ defines an isomorphism between the universal deformation rings of a supersingular representation $\bar{\sigma}$ of $L$ and of its parabolic induction $\bar{\pi}$.
As a consequence, we show that every Banach lift of $\bar{\pi}$ is induced from a unique Banach lift of $\bar{\sigma}$.
\end{abstract}

\maketitle

\tableofcontents

\section{Introduction}

Let $F/\Q_p$ be a finite extension, and let $G$ denote the $F$-points of a fixed connected reductive group defined over $F$.
In a recent paper (\cite{AHHV17}) Abe, Henniart, Herzig and Vignéras give a complete classification of the irreducible admissible $\overline{\F}_p$-representations $\bar{\pi}$ of $G$ in terms of the supersingular representations, which remain mysterious for groups other than $\GL_2(\Q_p)$ (and scattered rank one examples).
As a byproduct of the classification, supersingular is the same as supercuspidal (meaning it is not a subquotient of a representation induced from a proper parabolic subgroup).
Thus it forms the philosophical counterpart of Bernstein-Zelevinsky theory for $\GL_n$ in the complex case (\cite{BZ77}).
This finishes a long project initiated by Barthel and Livné in \cite{BL94}, who looked at $\GL_2$, and continued by Herzig (the case of $\GL_n$) and Abe (the split case) in \cite{Her11} and \cite{Abe13} respectively.

One feature of the emerging $p$-adic Langlands program (\cite{Bre10}) is a certain compatibility with the deformation theories on both sides (see \cite{Col10,Kis10,Pas13} for the established case $G=\GL_2(\Q_p)$).
In this paper we study deformations of representations $\bar{\pi}=\Ind_{P^-}^G \bar{\sigma}$ which are smoothly induced from an admissible representation $\bar{\sigma}$ of a Levi subgroup $L$.
Here $P=L \ltimes N$ is a parabolic subgroup with opposite $P^-$, and all these representations have coefficients in some fixed finite extension $k/\F_p$, which we take to be the residue field $k=\OO/(\varpi)$ of some fixed finite extension $E/\Q_p$.
Our hope is that our results will play a role in future developments of the $p$-adic Langlands program beyond the $\GL_2(\Q_p)$-case.

\medskip

The first part of our paper deals with deformations over Artinian rings, and forms the core of our argument.
Letting $\Art(\OO)$ denote the category of local Artinian $\OO$-algebras with residue field $k$, we consider the deformation functor $\Def_{\bar{\pi}}: \Art(\OO) \to \Set$ which takes $A \in \Art(\OO)$ to the set of equivalence classes of lifts $\pi$ of $\bar{\pi}$ over $A$.
Thus $\pi$ is a smooth $A[G]$-module, free over $A$, endowed with an isomorphism $\pi \otimes_A k \iso \bar{\pi}$.
Assuming $\End_{k[G]}(\bar{\pi})=k$, the functor $\Def_{\bar{\pi}}$ is known to be pro-representable, as recently shown by one of us (\cite{Sch13}).
To allow more flexibility one actually deforms the Pontrjagin dual $\bar{\pi}^{\vee} \coloneqq \Hom_k(\bar{\pi},k)$ which lives in a category of profinite augmented representations.
Let $R_{\bar{\pi}^{\vee}}$ be the universal deformation ring of $\bar{\pi}^{\vee}$ and $M_{\bar{\pi}^{\vee}}$ be the universal deformation of $\bar{\pi}^{\vee}$.
The duality transforms the parabolic induction functor into a functor $\Indd_{P^-}^G$ which yields a homomorphism $R_{\bar{\pi}^{\vee}}{\to } R_{\bar{\sigma}^{\vee}}$ of local profinite $\OO$-algebras.
The following is our main result.
Recall that a Banach lift of $\bar{\pi}$ is a unitary continuous $E$-Banach space representation of $G$ with a mod $\varpi$ reduction isomorphic to $\bar{\pi}$.

\begin{thm}\label{mainthm}
Let $\bar{\sigma}$ be an admissible smooth $k[L]$-module with $\End_{k[L]}(\bar{\sigma})=k$ and $\bar{\pi} \coloneqq \Ind_{P^-}^G \bar{\sigma}$.
If $F=\Q_p$, then assume that $\bar{\sigma}$ is supersingular and $\bar{\sigma}^\alpha \otimes (\bar{\varepsilon}^{-1} \circ \alpha) \not \simeq \bar{\sigma}$ for all $\alpha \in \Delta_L^{\perp,1}$.
Then the following hold.
\begin{enumerate}
\item $\Ind_{P^-}^G: \Def_{\bar{\sigma}}\to \Def_{\bar{\pi}}$ is an isomorphism of functors $\Art(\OO) \to \Set$.
\item Every Banach lift of $\bar{\pi}$ is induced from a unique Banach lift of $\bar{\sigma}$.
\item There is an isomorphism $R_{\bar{\pi}^{\vee}} \iso R_{\bar{\sigma}^{\vee}}$ through which $M_{\bar{\pi}^{\vee}} = \Indd_{P^-}^G (M_{\bar{\sigma}^{\vee}})$.
\item If $\dim_k \Ext_L^1(\bar{\sigma},\bar{\sigma}) < \infty$, then $R_{\bar{\pi}} \simeq R_{\bar{\sigma}}$ is Noetherian and there is an $R_{\bar{\pi}}[G]$-linear isomorphism $M_{\bar{\pi}} \simeq \Ind_{P^-}^G M_{\bar{\sigma}}$.
\end{enumerate}
\end{thm}

For the unexplained notation, let $B \subset G$ be a minimal parabolic subgroup contained in $P$ and let $S \subset B$ be a maximal split torus contained in $L$.
We denote by $\Delta$ the simple roots of the triple $(G,B,S)$ and by $\Delta^1$ the subset consisting of roots whose corresponding root space is one-dimensional.
If $\Delta_L \subset \Delta$ denotes the set of simple roots of the triple $(L,B \cap L,S)$, its orthogonal complement $\Delta_L^{\perp}$ is the set of roots $\alpha \in \Delta$ for which $\langle \alpha,\beta^{\vee}\rangle=0$ for all $\beta \in \Delta_L$.
Then $\Delta_L^{\perp,1} \coloneqq \Delta_L^\perp \cap \Delta^1$.
Given $\alpha \in \Delta_L^{\perp,1}$, we can therefore consider the smooth $L$-representation $\bar{\sigma}^\alpha \otimes (\bar{\varepsilon}^{-1} \circ \alpha)$ over $k$ where $\bar{\sigma}^\alpha$ is the $s_\alpha$-conjugate of $\bar{\sigma}$ and $\bar{\varepsilon} : F^\times \to k^\times$ is the reduction mod $p$ of the $p$-adic cyclotomic character (see Subsection \ref{grp} for more details).

Implicit in (3) and (4) is the fact that the functor $\Def_{\bar{\pi}}$ is also pro-representable by $R_{\bar{\pi}} = R_{\bar{\pi}^\vee}$, and when the latter is Noetherian there exists a universal deformation $M_{\bar{\pi}} = M_{\bar{\pi}^\vee}^\vee$ which is a continuous representation of $G$ over $R_{\bar{\pi}}$.
In this context, $\Ind_{P^-}^G$ refers to the continuous parabolic induction functor.

We refer to Theorem \ref{thm:defind} and Corollaries \ref{cor:Ban}, \ref{univ}, \ref{noeth} in the main text for more precise statements.
In particular, the assumption that $\bar{\sigma}$ is supersingular (imposed when $F=\Q_p$) can be weakened, cf. Hypothesis \ref{hyp:HOrd} (and the pertaining remark) which only requires that $\Ord_Q \bar{\sigma}=0$ for certain proper parabolic subgroups $Q \subset L$.
Here $\Ord_Q$ denotes Emerton's ordinary parts functor (\cite{Eme10a}), which is the right adjoint of $\Ind_{Q^-}^L$.
Moreover, (1) and (2) hold true without the assumption $\End_{k[L]}(\bar{\sigma})=k$, e.g. in the following cases:
\begin{itemize}
\item $F \neq \Q_p$ and $\bar{\sigma}$ is any admissible smooth $k[L]$-module,
\item $F = \Q_p$ and $\bar{\sigma} \simeq \bar{\sigma}_1 \oplus \dots \oplus \bar{\sigma}_r$ where the direct summands are supersingular and satisfy $\bar{\sigma}_i \not \simeq \bar{\sigma}_j^\alpha \otimes (\bar{\varepsilon}^{-1} \circ \alpha)$ for any $1 \leq i,j \leq r$ and any $\alpha \in \Delta_L^{\perp,1}$.
\end{itemize}

The genericity condition in Theorem \ref{mainthm} is rather mild (see Example \ref{exmp:gen} below), and it is sharp: for $G=\GL_2(\Q_p)$ the so-called ``atomes automorphes'' of length $2$ (\cite{Col10}) provide counter-examples of $\bar{\pi} = \Ind_{B^-}^G \bar{\chi}_1 \bar{\varepsilon}^{-1} \otimes \bar{\chi}_2 \oplus \Ind_{B^-}^G \bar{\chi}_2 \bar{\varepsilon}^{-1} \otimes \bar{\chi}_1$ (where $\bar{\chi}_1,\bar{\chi}_2 : \Q_p^\times \to k^\times$ are smooth characters) admitting topologically irreducible Banach lifts $\pi$ (thus not a direct sum of two unitary continuous principal series).

\begin{exmp} \label{exmp:gen}
If $G=\GL_n(\Q_p)$ and $L$ is a block-diagonal subgroup, then $\Delta^1=\Delta$ and the roots in $\Delta_L^\perp$ correspond to pairs of consecutive $\Q_p^\times$-factors of $L$.
Thus for each $\alpha \in \Delta_L^\perp$, one has a factorization as a direct product $L \simeq L' \times \Q_p^\times \times \Q_p^\times \times L''$, and if correspondingly $\bar{\sigma}$ decomposes as a tensor product $\bar{\sigma} \simeq \bar{\sigma}' \otimes \bar{\chi}_1 \otimes \bar{\chi}_2 \otimes \bar{\sigma}''$, then $\bar{\sigma}^\alpha \otimes (\bar{\varepsilon}^{-1} \circ \alpha) \simeq \bar{\sigma}' \otimes \bar{\chi}_2 \bar{\varepsilon}^{-1} \otimes \bar{\chi}_1 \bar{\varepsilon} \otimes \bar{\sigma}''$ so that the genericity condition becomes $\bar{\chi}_1 \bar{\chi}_2^{-1} \neq \bar{\varepsilon}^{-1}$.
\end{exmp}

In the principal series case, the result takes a more concrete form.
So assume $G$ is quasi-split and specialize to the case where $P=B$ is a Borel subgroup.
Then the Levi factor $L=T$ is a $p$-adic torus.
Let $T^{(p)} \coloneqq \varprojlim_j T/T^{p^j}$ denote the $p$-adic completion of $T$ and denote by $\Lambda \coloneqq \OO[[T^{(p)}]]$ its completed group algebra.

\begin{cor}
Let $\bar{\chi}: T \to k^\times$ be a smooth character and $\bar{\pi} \coloneqq \Ind_{B^-}^G \bar{\chi}$.
If $F=\Q_p$, then assume that $s_\alpha(\bar{\chi}) \cdot (\bar{\varepsilon}^{-1} \circ \alpha) \neq \bar{\chi}$ for all $\alpha \in \Delta^1$.
Then $R_{\bar{\pi}} \simeq \Lambda$ is Noetherian and $M_{\bar{\pi}} \simeq \Ind_{B^-}^G \chi^\univ$.
\end{cor}

Here, $\chi^\univ: T \to \Lambda^\times$ denotes the universal deformation of $\bar{\chi}$.

\medskip

Let us briefly sketch the further content of this paper.
In the second part, we study the deformation theory of parabolic induction over complete local Noetherian rings by passing the Artinian theory, so to speak, to the limit.
On the way, we establish several properties of the continuous parabolic induction functor.
The dimension of the tangent space $\Ext_L^1(\bar{\sigma},\bar{\sigma})$ of a deformation ring of type $R_{\bar{\sigma}^{\vee}}$ is not easily accessible.
For example it is not known whether it is finite-dimensional for a supersingular representation $\bar{\sigma}$, except when $L$ is a torus or $L=\GL_2(\Q_p)$ (\cite{Pas10}).
This forces us to go one step further and work over quite arbitrary profinite rings.
This forms the topic of the third part which contains the main result and its proof.

\medskip

The origin of this article is a paper of one of us (\cite{Sor15}) which dealt with the case of principal series using the calculations of \cite{Hau16a}.
Meanwhile, these calculations were generalized in \cite{Hau16b} and the three authors decided to extend the results of the original paper in order to treat the general case.
Since the first version of this article, some calculations have been carried over a base field of characteristic $p$ (\cite{Hau17}) allowing our main result to be generalized \emph{verbatim} to any non-archimedean local field $F$ of residue characteristic $p$.
Finally, we point out a sequel to this article (\cite{HSS17}) in which we compute the deformations of generalized Steinberg representations.

\subsection{Notation and conventions}

\subsubsection{Coefficient algebras}

Throughout the paper we fix a finite extension $E/\Q_p$ which will serve as our coefficient field.
We denote its integer ring by $\OO$ and we fix a uniformizer $\varpi \in \OO$.
The residue field $k=\OO/\varpi\OO$ is a finite field of cardinality $q$.
The normalized absolute value on $E$ is denoted by $|\cdot|$; thus $|\varpi|=q^{-1}$.
We write $\bar{\varepsilon} : \Q_p^\times \to \F_p^\times \subset k$ for the reduction mod $p$ of the $p$-adic cyclotomic character.

We write $\Art(\OO)$ for the category whose objects are local Artinian $\OO$-algebras $A$ (with the discrete topology) for which the structure map $\OO \to A$ is local and induces an isomorphism $k \iso A/\m_A$.
The morphisms $A \to A'$ are the (local) $\OO$-algebra homomorphisms.
Note that $k$ is a terminal object of $\Art(\OO)$.

\begin{rem} \label{length}
Note that any $A \in \Art(\OO)$ has finite $\OO$-length.
In fact the $\OO$-length $\ell_{\OO}(A)$ equals the $A$-length $\ell_A(A)$: any simple $A$-module is isomorphic to $A/\m_A\simeq k$ and, hence, is a simple $\OO$-module.
So a composition series for $A$ as an $A$-module is a composition for $A$ as an $\OO$-module.
We will drop subscripts and just write $\ell(A)$.
\end{rem}

We write $\Noe(\OO)$ for the category whose objects are Noetherian complete local $\OO$-algebras $A$ for which the structural morphism $\OO \to A$ is local and induces an isomorphism $k \iso A/\m_A$.
The morphisms $A \to A'$ are the local $\OO$-algebra homomorphisms.
Note that $\Art(\OO)$ is the full subcategory of $\Noe(\OO)$ consisting of Artinian rings, and that $A/\m_A^n \in \Art(\OO)$ for all $n \geq 1$ when $A \in \Noe(\OO)$.

We write $\Pro(\OO)$ for the category whose objects are local profinite $\OO$-algebras $A$ for which the structure map $\OO \to A$ is local and induces an isomorphism $k \iso A/\m_A$.
The morphisms $A \to A'$ are the continuous local $\OO$-algebra homomorphisms.
Note that $\Art(\OO)$ is the full subcategory of $\Pro(\OO)$ consisting of Artinian rings, and that $A/\af \in \Art(\OO)$ when $A \in \Pro(\OO)$ and $\af \subset A$ is an open ideal.
Furthermore, $\Pro(\OO)$ is equivalent to the category of pro-objects of $\Art(\OO)$ (cf. \cite[Lem.~3.3]{Sch13}).
If $A \in \Pro(\OO)$ is Noetherian, then the profinite topology is the $\m_A$-adic topology (\cite[Prop.~22.5]{Sch11}).
Moreover, a morphism $A \to A'$ between Noetherian rings in $\Pro(\OO)$ is continuous if and only if it is local.
Thus $\Noe(\OO)$ is the full subcategory of $\Pro(\OO)$ consisting of Noetherian rings.\footnote{To give a simple example, the ring of dual numbers $k[\{\epsilon_i\}_{i\in\N}]$ in an infinite number of parameters $\epsilon_i$, where $\epsilon_i\epsilon_j=0$ for all $i,j$,  equals the inverse limit over the finite rings $k[\epsilon_1,...,\epsilon_n]$, i.e. lies in $\Pro(\OO)$, but it is not Noetherian. Another example is the ring of formal power series over $\OO$ in countably infinitely many indeterminates $X_1,X_2,X_3\dots$.}

\subsubsection{Reductive $p$-adic groups} \label{grp}

We fix a finite extension $F/\Q_p$ which will serve as our base field.
We let $G$ be a connected reductive group over $F$.
By abuse of notation, instead of $G(F)$ we simply write $G$.
The same convention applies to other linear algebraic $F$-groups.

We choose a minimal parabolic subgroup $B \subset G$ and a maximal split torus $S \subset B$.
We let $\ZZ$ be the centralizer of $S$ in $G$, $\NN$ be its normalizer, and $W=\ZZ/\NN$ be the Weyl group of $(G,S)$.
We let $\Delta$ denote the set of simple roots of the triple $(G,B,S)$.
For $\alpha \in \Delta$ we denote by $\gf_{(\alpha)} \coloneqq \gf_\alpha \oplus \gf_{2\alpha}$ the corresponding subspace in the Lie algebra of $G$ (with the convention that $\gf_{2\alpha}=0$ if $2\alpha$ is not a root).
We put
\begin{equation*}
\Delta^1 \coloneqq \left\{ \alpha \in \Delta \middlevert \dim_F \gf_{(\alpha)} = 1 \right\}.
\end{equation*}
We have $\Delta^1=\Delta$ if $G$ is split, but not in general (even if the root system of $(G,S)$ is reduced, e.g. $\Delta^1=\varnothing$ if $G = \Res_{F'/F} G'$ with $F'/F$ strict).

We fix a standard parabolic subgroup $P \supset B$ and let $L \supset S$ be the standard Levi factor.
We denote by $P^-$ the opposite parabolic with respect to $L$, i.e. the unique parabolic subgroup such that $P \cap P^-=L$, and we write $Z_L$ for the center of $L$.
Similarly we let $\Delta_L \subset \Delta$ denote the set of simple roots of the triple $(L,B \cap L,S)$ and its orthogonal complement $\Delta_L^{\perp}$ is the set of roots $\alpha \in \Delta$ for which $\langle \alpha,\beta^{\vee}\rangle=0$ for all $\beta \in \Delta_L$.
For example $\Delta_G^{\perp}=\varnothing$ and at the other extreme $\Delta_\ZZ^{\perp}=\Delta$.

Finally, we put $\Delta_L^{\perp,1} \coloneqq \Delta_L^\perp \cap \Delta^1$.
For $\alpha \in \Delta_L^{\perp,1}$, conjugation by a representative $n_\alpha \in \NN$ of the corresponding simple reflection $s_\alpha \in W$ stabilizes $L$, and $\alpha$ extends (uniquely) to an algebraic character of $L$ (cf. the proof of \cite[Lem.~5.1.4]{Hau16b}).
Therefore if $\bar{\sigma}$ is a smooth $k[L]$-module and $F=\Q_p$, we can consider the smooth $k[L]$-module $\bar{\sigma}^\alpha \otimes (\bar{\varepsilon}^{-1} \circ \alpha)$ where $\bar{\sigma}^\alpha$ is the $s_\alpha$-conjugate of $\bar{\sigma}$ (i.e. $\bar{\sigma}^\alpha$ has the same underlying $k$-vector space as $\bar{\sigma}$, but $l \in L$ acts on $\bar{\sigma}^\alpha$ as $n_\alpha l n_\alpha^{-1}$ acts on $\bar{\sigma}$), which does not depend on the choice of $n_\alpha$ in $n_\alpha \ZZ$ up to isomorphism (since $\ZZ \subset L$).

\begin{exmp}
We find it instructive to unravel the notation in the case of $G=\GL_n(\Q_p)$, where we take $B$ to be the upper-triangular matrices and $S$ to be the diagonal matrices.
In this case $\Delta = \{ e_i-e_{i+1} : 1 \leq i < n \}$ where $e_i : S \to \Q_p^\times$ is the algebraic character defined by $t = \diag(t_1,\dots,t_n) \mapsto t_i$.

Then $L \simeq \GL_{n_1}(\Q_p) \times \dots \times \GL_{n_r}(\Q_p)$ is a block-diagonal subgroup, and $\Delta_L = \{ e_i-e_{i+1} : n_1+\dots+n_j+1 \leq i < n_1+\dots+n_{j+1} \text{ for some } 0 \leq j < r \}$.
Loosely speaking the roots in $\Delta \backslash \Delta_L$ give the consecutive ratios where the blocks of $L$ meet.

The roots in $\Delta_L^{\perp}$ correspond to pairs of consecutive $\Q_p^\times$-factors of $L$: $\alpha \in \Delta_L^\perp$ if and only if $\alpha=e_i-e_{i+1}$ with $0 \leq i <n$ such that $n_i=n_{i+1}=1$.
In this case, conjugation by $s_\alpha$ permutes the corresponding two copies of $\Q_p^\times$ and $\alpha$ extends to an algebraic character $L \to \Q_p^\times$ giving the ratio between them.

Therefore if $\bar{\sigma}$ is a smooth $k[L]$-module such that $\bar{\sigma} \simeq \bar{\sigma}_1 \otimes \dots \otimes \bar{\sigma}_r$ where each $\bar{\sigma}_j$ is a smooth $k[\GL_{n_j}(\Q_p)]$-module, and $\alpha=e_i-e_{i+1} \in \Delta_L^\perp$, then $\bar{\sigma}^\alpha \otimes (\bar{\varepsilon}^{-1}\circ \alpha) \simeq \bar{\sigma}_1 \otimes \dots \otimes \bar{\sigma}_{i+1} \bar{\varepsilon}^{-1} \otimes \bar{\sigma}_i \bar{\varepsilon} \otimes \dots \otimes \bar{\sigma}_r$.
\end{exmp}

\section{Parabolic induction and deformations over Artinian rings}

\subsection{Smooth parabolic induction}

Let $A \in \Art(\OO)$.
We consider the category of all $A[G]$-modules $\Mod_G(A)$, and its full abelian subcategory of smooth representations $\Mod_G^\infty(A)$.
Recall that $\pi$ is smooth if $\pi = \bigcup_K \pi^K$ where $K \subset G$ ranges over compact open subgroups.
We say that $\pi$ is admissible if each $\pi^K$ is a finitely generated $A$-module.
The admissible representations form a Serre subcategory $\Mod_G^{\adm}(A)$ (cf. \cite[Prop.~2.2.13]{Eme10a}).
Finally we let $\Mod_G^\infty(A)^\fl$ be the full subcategory of $\Mod_G^\infty(A)$ consisting of objects free over $A$.

The parabolic induction of a smooth $A[L]$-module $\sigma$ is defined as follows.
First inflate $\sigma$ via the projection $P^- \twoheadrightarrow L$ and let
\begin{equation*}
\Ind_{P^-}^G \sigma \coloneqq \left\{ \text{smooth $f : G \to \sigma$} \middlevert \text{$f(pg) = pf(g)$ for all $p \in P^-$ and $g \in G$} \right\}.
\end{equation*}
Smoothness of $f$ means continuous relative to the discrete topology on $\sigma$ (i.e. locally constant).
Thus $\Ind_{P^-}^G \sigma$ becomes a smooth $A[G]$-module via right translations, and this defines an $A$-linear functor
\begin{equation*}
\Ind_{P^-}^G: \Mod_L^\infty(A) \longrightarrow \Mod_G^\infty(A)
\end{equation*}
which is exact, commutes with small direct sums and preserves admissibility (cf. Lem.~4.1.4 and Prop.~4.1.5 in \cite{Eme10a}).

\begin{lem} \label{lem:Ind}
Let $\sigma$ be a smooth $A[L]$-module.
\begin{enumerate}
\item $\Ind_{P^-}^G \sigma$ is free over $A$ if and only if $\sigma$ is free over $A$.
\item For any morphism $A \to A'$ in $\Art(\OO)$, there is a natural $A'[G]$-linear isomorphism
\begin{equation*}
\left( \Ind_{P^-}^G \sigma \right) \otimes_A A' \iso \Ind_{P^-}^G \left( \sigma \otimes_A A' \right).
\end{equation*}
\end{enumerate}
\end{lem}

\begin{proof}
There are natural $A$-linear isomorphisms (cf. \cite[(6)]{Vig16})
\begin{equation} \label{CC}
\Ind_{P^-}^G \sigma \simeq \CC^\infty \left( P^-\backslash G, \sigma \right) \simeq \CC^\infty \left( P^- \backslash G, A \right) \otimes_A \sigma
\end{equation}
(the first one is induced by composition with a continuous section of the projection $G \twoheadrightarrow P^- \backslash G$, and the second one follows from the fact that a smooth function $f : P^- \backslash G \to \sigma$ takes only finitely many values).
Furthermore, $\CC^\infty(P^-\backslash G,A)$ is a direct limit of finite-free $A$-modules
\begin{equation*}
\CC^\infty \left( P^- \backslash G, A \right) \simeq \varinjlim_K \CC^\infty \left( P^- \backslash G/K, A \right)
\end{equation*}
where $K \subset G$ runs through the compact open subgroups, thus it is flat and therefore free.
In conjunction with \eqref{CC}, this immediately shows that $\Ind_{P^-}^G \sigma$ is free over $A$ if $\sigma$ is.
For the converse note that $\sigma$ is a direct summand of $\Ind_{P^-}^G \sigma$.
If the latter is free $\sigma$ is projective, which is the same as free over local rings (by Kaplansky's theorem).
This shows (1).
For the proof of (2), note that the natural map defined by $f \otimes a' \mapsto (g \mapsto f(g) \otimes a')$ is $A[G]$-linear and using \eqref{CC} twice (with $\sigma$ and $\sigma \otimes_A A'$), we see that it is bijective (both sides being naturally isomorphic to $\CC^\infty(P^- \backslash G, A) \otimes_A (\sigma \otimes_A A')$).
\end{proof}

\subsection{Ordinary parts}

For convenience we briefly recall the basic properties of Emerton's functor of ordinary part $\Ord_P$, which is somewhat analogous to (but better behaved than) the locally analytic Jacquet functor $J_P$.
Let $A \in \Art(\OO)$.
We let $\Mod_L^{\infty,Z_L-\lfin}(A)$ be the full subcategory of locally $Z_L$-finite\footnote{I.e. the $A[Z_L]$-submodule generated by any element is of finite type over $A$ (cf. \cite[Def.~3.2.1 (3)]{Eme10a}).} smooth $A[L]$-modules.
Then the functor of ordinary part is an $A$-linear functor
\begin{equation*}
\Ord_P : \Mod_G^\infty(A) \to \Mod_L^{\infty,Z_L-\lfin}(A)
\end{equation*}
which is left-exact, commutes with small inductive limits and preserves admissibility (cf. Prop.~3.2.4 and Thm.~3.3.3 in \cite{Eme10a}).
Furthermore, it is related to the smooth parabolic induction functor by the following result.

\begin{thm}[Emerton] \label{thm:Ord}
Let $\sigma$ be a locally $Z_L$-finite smooth $A[L]$-module.
\begin{enumerate}
\item There is a natural $A[L]$-linear isomorphism
\begin{equation*}
\sigma \iso \Ord_P \left( \Ind_{P^-}^G \sigma \right).
\end{equation*}
\item For any smooth $A[G]$-module $\pi$, $\Ord_P$ induces an $A$-linear isomorphism
\begin{equation*}
\Hom_{A[G]} \left( \Ind_{P^-}^G \sigma, \pi \right) \iso \Hom_{A[L]} \left( \sigma, \Ord_P \pi \right).
\end{equation*}
\end{enumerate}
\end{thm}

\begin{proof}
This is Prop.~4.3.4 and essentially Thm.~4.4.6 in \cite{Eme10a} respectively, except that in op. cit. the latter is only formulated for an admissible representation $\sigma$.
However, the proof only uses the admissibility hypothesis when invoking \cite[Lem.~4.4.3]{Eme10a} in order to prove that (4.4.7) of op. cit. is surjective.
But this fact is proved by Vignéras for a merely locally $Z_L$-finite representation $\sigma$ (cf. the proof of \cite[Thm.~6.1]{Vig16}).
\end{proof}

In other words, $\Ord_P$ is a left quasi-inverse and the right adjoint of $\Ind_{P^-}^G$.
Since both functors respect admissibility, their restrictions to the corresponding full subcategories of admissible representations (note that an admissible representation of $L$ is locally $Z_L$-finite by \cite[Lem.~2.3.4]{Eme10a}) are still adjoint to each other.

\begin{rem} \label{rem:R}
Vignéras (\cite{Vig16}) proved that the functor $\Ind_{P^-}^G$ between the categories of smooth representations also admits a right adjoint $\operatorname{R}_{P^-}^G$.
When restricted to the categories of admissible representations, $\operatorname{R}_{P^-}^G$ is isomorphic to $\Ord_P$ (\cite[Cor.~4.13]{AHV17}).
\end{rem}

We rephrase the adjunction relation between $\Ind_{P^-}^G$ and $\Ord_P$ in a more formal but equivalent way (cf. \cite[Prop.~1.5.4]{KS06}), which will be convenient for the proofs of the results in Section \ref{id}.
We use the isomorphisms in Theorem \ref{thm:Ord}:
For any object $\sigma$ of $\Mod_L^\adm(A)$ and $\pi \coloneqq \Ind_{P^-}^G \sigma$, the natural $A[L]$-linear isomorphism in (1) is the image $\epsilon_\sigma$ of $\id_\pi$ under the isomorphism in (2).
Likewise for any object $\pi$ of $\Mod_G^\adm(A)$ and $\sigma \coloneqq \Ord_P \pi$, the preimage $\eta_\pi$ of $\id_\sigma$ under the isomorphism in (2) is a natural $A[G]$-linear morphism $\Ind_{P^-}^G \Ord_P \pi \to \pi$.
Hence, there are natural transformations (called the unit and counit respectively)
\begin{gather*}
\epsilon : \id_{\Mod^\adm_L(A)} \iso \Ord_P \Ind_{P^-}^G \\
\eta : \Ind_{P^-}^G \Ord_P \longrightarrow \id_{\Mod^\adm_G(A)}
\end{gather*}
which satisfy the following equalities (cf. (1.5.8) and (1.5.9) in \cite{KS06})
\begin{gather}
\left( \eta \Ind_{P^-}^G \right) \left( \Ind_{P^-}^G \epsilon \right) = \id_{\Ind_{P^-}^G} \label{eqadj} \\
\left( \Ord_P \eta \right) \left( \epsilon \Ord_P \right) = \id_{\Ord_P} \notag
\end{gather}
(the compositions inside the parentheses are compositions of natural transformations with functors, whereas the compositions outside the parentheses are ``vertical compositions'' of natural transformations, so that they yield composites of natural transformations
\begin{gather*}
\Ind_{P^-}^G \iso \Ind_{P^-}^G \Ord_P \Ind_{P^-}^G \longrightarrow \Ind_{P^-}^G \\
\Ord_P \iso \Ord_P \Ind_{P^-}^G \Ord_P \longrightarrow \Ord_P
\end{gather*}
respectively).

\subsection{Higher ordinary parts}

Let $A \in \Art(\OO)$.
In \cite{Eme10b}, Emerton constructed a cohomological $\delta$-functor
\begin{equation*}
\HOrd_P : \Mod_G^\adm(A) \to \Mod_L^\adm(A)
\end{equation*}
which coincides with $\Ord_P$ in degree $0$.
This means that a short exact sequence
\begin{equation} \label{SEC}
0 \to \pi' \to \pi \to \pi'' \to 0
\end{equation}
in $\Mod_G^\adm(A)$ gives rise to a long exact sequence
\begin{equation*}
\cdots \to \HOrd[n]_P \pi' \to \HOrd[n]_P \pi \to \HOrd[n]_P \pi'' \to \HOrd[n+1]_P \pi' \to \cdots
\end{equation*}
with degree-increasing connecting homomorphisms functorial in \eqref{SEC}.

We review a recent result of one of us (J.H.) which gives a complete description of $\HOrd[1]_P(\Ind_{P^-}^G \bar{\sigma})$ for any admissible smooth $k[L]$-module $\bar{\sigma}$ (satisfying a technical assumption when $F=\Q_p$).
This will play a key role in Section \ref{id} in the proof of Theorem \ref{thm:defind}.

Let $\sigma$ be an admissible smooth $k[L]$-module.
We consider the following hypothesis, which will be needed when $F=\Q_p$.
Note that for all $\alpha \in \Delta \backslash \Delta_L$, $L \cap s_\alpha P s_\alpha$ is the standard parabolic subgroup of $L$ corresponding to $\Delta_L \cap s_\alpha (\Delta_L)$ and it is proper if and only if $\alpha \not \in \Delta_L^\perp$.

\begin{hyp} \label{hyp:HOrd}
$\Ord_{L \cap s_\alpha P s_\alpha} \sigma = 0$ for all $\alpha \in \Delta^1 \backslash (\Delta_L^1 \cup \Delta_L^{\perp,1})$.
\end{hyp}

\begin{rem}
In the terminology of \cite{AHV17}, $\sigma$ is said right cuspidal if $\Ord_Q \sigma = 0$ for any proper parabolic subgroup $Q \subset L$.
Thus Hypothesis \ref{hyp:HOrd} is satisfied by any right cuspidal representation $\sigma$.
In particular, it is satisfied by any supersingular representation $\sigma$.
\end{rem}

We now state the key calculation.

\begin{thm}[J.H.] \label{thm:HOrd}
Let $\bar{\sigma}$ be an admissible smooth $k[L]$-module.
\begin{enumerate}
\item If $F = \Q_p$ and $\bar{\sigma}$ satisfies Hypothesis \ref{hyp:HOrd}, then there is a natural $k[L]$-linear isomorphism
\begin{equation*}
\HOrd[1]_P \left( \Ind_{P^-}^G \bar{\sigma} \right) \simeq \bigoplus_{\alpha \in \Delta_L^{\perp,1}} \bar{\sigma}^\alpha \otimes \left( \bar{\varepsilon}^{-1} \circ \alpha \right).
\end{equation*}
\item If $F \neq \Q_p$, then $\HOrd[1]_P (\Ind_{P^-}^G \bar{\sigma}) = 0$.
\end{enumerate}
\end{thm}

\begin{proof}
This is \cite[Cor.~3.3.9]{Hau16b} with $A=k$ and $n=1$.
\end{proof}

\subsection{Deformation functors}

We let $\bar{\pi}$ be a smooth $k[G]$-module.
We will study the various lifts of $\bar{\pi}$ to smooth representations $\pi$ of $G$ over $A\in \Art(\OO)$.
In what follows $\Cat$ denotes the category of essentially small\footnote{A category $C$ is essentially small if it is equivalent to a small category (cf. \cite[Def.~1.3.16]{KS06}), i.e. if the isomorphism classes $\Ob(C)/\simeq$ form a set.} categories and $i : \Cat \to \Set$ denotes the functor taking an essentially small category $C$ to the set of isomorphism classes $\Ob(C)/\simeq$.

\begin{defn}\label{def}
We define several categories and functors.
\begin{enumerate}
\item A lift of $\bar{\pi}$ over $A \in \Art(\OO)$ is a pair $(\pi,\phi)$ where
\begin{itemize}
\item $\pi$ is an object of $\Mod_G^\infty(A)^\fl$,
\item $\phi: \pi \twoheadrightarrow \bar{\pi}$ is an $A[G]$-linear surjection with kernel $\m_A \pi$, i.e. which induces an $A[G]$-linear isomorphism $\pi \otimes_A k \iso \bar{\pi}$.
\end{itemize}
A morphism $\iota : (\pi,\phi) \to (\pi',\phi')$ of lifts of $\bar{\pi}$ over $A$ is an $A[G]$-linear morphism $\iota : \pi \to \pi'$ such that $\phi = \phi' \circ \iota$.
\item We define a covariant functor $\Deft_{\bar{\pi}} : \Art(\OO) \to \Cat$ by letting $\Deft_{\bar{\pi}}(A)$ be the essentially small category of lifts of $\bar{\pi}$ over $A$ for any $A \in \Art(\OO)$, and $\Deft_{\bar{\pi}}(\varphi) : \Deft_{\bar{\pi}}(A) \to \Deft_{\bar{\pi}}(A')$ be the base change functor for any morphism $\varphi : A \to A'$ in $\Art(\OO)$.
\item We define the deformation functor $\Def_{\bar{\pi}} : \Art(\OO) \to \Set$ as the composite $i \circ \Deft_{\bar{\pi}}$.
\end{enumerate}
\end{defn}

\begin{rem}
$\widetilde{\Def}_{\bar{\pi}}(k)$ is a groupoid, and $\Def_{\bar{\pi}}(k)$ is a singleton.
\end{rem}

We review some properties of lifts of $\bar{\pi}$.

\begin{lem} \label{lem:lifts}
Let $\pi$ be a smooth $A$-representation of $G$ and $\bar{\pi} \coloneqq \pi/\m_A\pi$.
\begin{enumerate}
\item $\pi$ is locally $Z$-finite if and only if $\bar{\pi}$ is locally $Z$-finite.
\item $\pi$ is admissible if and only if $\bar{\pi}$ is admissible.
\end{enumerate}
\end{lem}

\begin{proof}
We proceed by induction on the length of $A$.
The base case $A=k$ is trivial.
Assume $A \neq k$ and that we know the results for rings of smaller length.
Pick $a \in A$ non-zero such that $a\m_A=0$ and set $A' \coloneqq A/aA$, so that $\ell(A') = \ell(A)-1$.
We set $\pi' \coloneqq \pi/a\pi$.
The surjection $\pi \twoheadrightarrow \pi'$ induces an isomorphism $\bar{\pi} \iso \pi'/\m_{A'}\pi'$.
The multiplication by $a$ induces an isomorphism $\bar{\pi} \iso a\pi$, hence a short exact sequence of smooth $A$-representations of $G$
\begin{equation*}
0 \to \bar{\pi} \to \pi \to \pi' \to 0.
\end{equation*}

If $\pi$ is locally $Z$-finite, then $\bar{\pi}$ is also locally $Z_M$-finite since it is a quotient of $\pi$.
Conversely, assume that $\bar{\pi}$ is locally $Z$-finite and let $v \in \pi$.
Since $\pi$ is smooth, there exists an open subgroup $Z_0 \subseteq Z$ fixing $v$.
Since $Z/Z_0$ is finitely generated as a monoid (see e.g. the proof of \cite[Lem.~3.2.1]{Eme10a}), there exist $r \in \N$ and $z_1,\dots,z_r \in Z$ such that $A[Z] \cdot v = A[z_1,\dots,z_r] \cdot v$.
Let $v'$ be the image of $v$ in $\pi'$.
By the induction hypothesis, $\pi'$ is locally $Z$-finite.
Thus for each $1 \leq i \leq r$, there exists $f'_i \in A[X]$ such that $f'_i(z_i) \cdot v' = 0$, i.e. $f'_i(z_i) \cdot v \in \bar{\pi}$.
Since $\bar{\pi}$ is locally $Z$-finite, there exists $\bar{f}_i \in A[X]$ such that $\bar{f}_i(z_i) \cdot (f'_i(z_i) \cdot v) = 0$.
We set $f_i \coloneqq \bar{f}_i f'_i \in A[X]$ so that $f_i(z_i) \cdot v = 0$.
Therefore $A[Z] \cdot v = A[z_1,\dots,z_r]/(f_1(z_1),\dots,f_r(z_r)) \cdot v$ is a finitely generated $A$-module.
This proves (i).

If $\pi$ is admissible, then $\bar{\pi}$ is also admissible since it is isomorphic to a subrepresentation of $\pi$ and $A$ is noetherian.
Conversely if $\bar{\pi}$ is admissible, then $\pi'$ is admissible by the induction hypothesis, so that $\pi$ is also admissible since it is an extension of $\pi'$ by $\bar{\pi}$.
This proves (ii).
\end{proof}

\subsection{Induced deformations}\label{id}

We let $\bar{\sigma}$ be a smooth $k[L]$-module and we set $\bar{\pi} \coloneqq \Ind_{P^-}^G \bar{\sigma}$.
By functoriality and using part (1) of Lemma \ref{lem:Ind}, we see that
\begin{itemize}
\item if $(\sigma,\phi)$ is a lift of $\bar{\sigma}$ over $A \in \Art(\OO)$, then $(\Ind_{P^-}^G \sigma,\Ind_{P^-}^G \phi)$ is a lift of $\bar{\pi}$ over $A$,
\item and if $\iota : (\sigma,\phi) \to (\sigma',\phi')$ is a morphism of lifts of $\bar{\sigma}$ over $A \in \Art(\OO)$, then $\Ind_{P^-}^G \iota : (\Ind_{P^-}^G \sigma,\Ind_{P^-}^G \phi) \to (\Ind_{P^-}^G \sigma',\Ind_{P^-}^G \phi')$ is a morphism of lifts of $\bar{\pi}$ over $A$.
\end{itemize}
Thus for any $A \in \Art(\OO)$, we obtain a functor
\begin{equation*}
\Ind_{P^-}^G : \Deft_{\bar{\sigma}}(A) \to \Deft_{\bar{\pi}}(A)
\end{equation*}
which is functorial in $A$ by part (2) of Lemma \ref{lem:Ind}, i.e. it is induced by a morphism $\Ind_{P^-}^G : \Deft_{\bar{\sigma}} \to \Deft_{\bar{\pi}}$ between functors $\Art(\OO) \to \Cat$.
Thus composing the latter with $i : \Cat \to \Set$ yields a morphism
\begin{equation*}
\Ind_{P^-}^G : \Def_{\bar{\sigma}} \to \Def_{\bar{\pi}}
\end{equation*}
between functors $\Art(\OO) \to \Set$.

\begin{lem} \label{lem:defind}
Let $\bar{\sigma}$ be a locally $Z_L$-finite smooth $k[L]$-module and $\bar{\pi} \coloneqq \Ind_{P^-}^G \bar{\sigma}$.
Then
\begin{enumerate}
\item $\Ind_{P^-}^G : \Deft_{\bar{\sigma}}(A) \to \Deft_{\bar{\pi}}(A)$ is fully faithful for any $A \in \Art(\OO)$,
\item $\Ind_{P^-}^G : \Def_{\bar{\sigma}}(A) \to \Def_{\bar{\pi}}(A)$ is injective for any $A \in \Art(\OO)$.
\end{enumerate}
\end{lem}

\begin{proof}
Since a fully faithful functor induces an injection between isomorphism classes of objects, (2) is a consequence of (1) which we now prove.

Let $A \in \Art(\OO)$.
Let $(\sigma,\phi)$ and $(\sigma',\phi')$ be two lifts of $\bar{\sigma}$ over $A$.
Note that $\sigma$ and $\sigma'$ are objects of $\Mod_L^{\infty,Z_L-\lfin}(A)$ by part (1) of Lemma \ref{lem:lifts} and recall the unit $\epsilon$ of the adjunction between $\Ind_{P^-}^G$ and $\Ord_P$.
We consider the $A$-linear morphism
\begin{equation} \label{HomInd}
\Hom_{\Deft_{\bar{\sigma}}(A)} \left( \left( \sigma, \phi \right), \left( \sigma', \phi' \right) \right) \to \Hom_{\Deft_{\bar{\pi}}(A)} \left( \Ind_{P^-}^G \left( \sigma,\phi \right), \Ind_{P^-}^G \left( \sigma', \phi' \right) \right)
\end{equation}
defined by $\imath \mapsto \Ind_{P^-}^G \imath$.
We claim that the map $\jmath \mapsto \epsilon_{\sigma'}^{-1} \circ (\Ord_P \jmath) \circ \epsilon_{\sigma}$ is an inverse.
\begin{itemize}
\item By naturality of $\epsilon$, it is well defined and $\epsilon_{\sigma'}^{-1} \circ (\Ord_P (\Ind_{P^-}^G \imath)) \circ \epsilon_{\sigma} = \imath$ for any morphism $\imath : (\sigma,\phi) \to (\sigma',\phi')$ in $\Deft_{\bar{\sigma}}(A)$.
\item By \eqref{eqadj} we also have $\Ind_{P^-}^G(\epsilon_{\sigma'}^{-1} \circ (\Ord_P \jmath) \circ \epsilon_{\sigma}) = \jmath$ for any morphism $\jmath : (\Ind_{P^-}^G \sigma,\Ind_{P^-}^G \phi) \to (\Ind_{P^-}^G \sigma',\Ind_{P^-}^G \phi')$ in $\Deft_{\bar{\pi}}(A)$.
\end{itemize}
Thus \eqref{HomInd} is bijective.
\end{proof}

\begin{rem} \label{Vig}
Using Vignéras' functor $\operatorname{R}_{P^-}^G$ (cf. Remark \ref{rem:R}) instead of $\Ord_P$, one can prove Lemma \ref{lem:defind} for any smooth $k[L]$-module $\bar{\sigma}$.
\end{rem}

The following is our main result in the Artinian case.
Our proof is an extension (and to some extent a correction\footnote{The inductive step fails if his $V$ is just faithful, not faithfully flat, and $\bar{\varepsilon}$ there should be $\bar{\varepsilon}^{-1}$.}) of the proof of \cite[Prop.~4.2.14]{Eme10b} (see also the pertaining Remark 4.2.17 in op. cit.)

\begin{thm} \label{thm:defind}
Let $\bar{\sigma}$ be an admissible smooth $k[L]$-module and $\bar{\pi} \coloneqq\Ind_{P^-}^G \bar{\sigma}$.
If $F=\Q_p$, then assume that
\begin{enumerate}[(a)]
\item Hypothesis \ref{hyp:HOrd} is satisfied,
\item $\Hom_L(\bar{\sigma},\bar{\sigma}^\alpha \otimes (\bar{\varepsilon}^{-1} \circ \alpha))=0$ for all $\alpha \in \Delta_L^{\perp,1}$.
\end{enumerate}
Then
\begin{enumerate}
\item $\Ind_{P^-}^G : \Deft_{\bar{\sigma}}(A) \to \Deft_{\bar{\pi}}(A)$ is an equivalence for any $A \in \Art(\OO)$ with quasi-inverse induced by $\Ord_P$,
\item $\Ind_{P^-}^G : \Def_{\bar{\sigma}} \to \Def_{\bar{\pi}}$ is an isomorphism.
\end{enumerate}
\end{thm}

\begin{proof}
Since an equivalence of categories induces a bijection between isomorphism classes of objects, (2) is a consequence of (1) which we now prove.
Given lemma \ref{lem:defind} (see also Remark \ref{Vig}), we only have to prove that $\Ind_{P^-}^G : \Deft_{\bar{\sigma}}(A) \to \Deft_{\bar{\pi}}(A)$ is essentially surjective for all $A \in \Art(\OO)$.

We proceed by induction on the length of $A \in \Art(\OO)$ (cf. Remark \ref{length}).
The base case $A=k$ is trivial.
Assume $\m_A \neq 0$ and that we know surjectivity for rings of smaller length.
Let $(\pi,\phi)$ be a lift of $\bar{\pi}$ over $A$.
Note that $\pi$ is an object of $\Mod_G^\adm(A)$ by part (2) of Lemma \ref{lem:lifts} and recall the unit $\epsilon$ and the counit $\eta$ of the adjunction between $\Ind_{P^-}^G$ and $\Ord_P$.
We will prove that $(\Ord_P \pi,\epsilon_{\bar{\sigma}}^{-1} \circ (\Ord_P \phi))$ is a lift of $\bar{\sigma}$ over $A$, and that the natural morphism of lifts of $\bar{\pi}$ over $A$
\begin{equation*}
\eta_\pi : \left( \Ind_{P^-}^G \left( \Ord_P \pi \right),\Ind_{P^-}^G \left( \epsilon_{\bar{\sigma}}^{-1} \circ \left( \Ord_P \phi \right) \right) \right) \to \left(\pi,\phi\right),
\end{equation*}
where we used the equality $\Ind_{P^-}^G (\epsilon_{\bar{\sigma}}^{-1} \circ (\Ord_P \phi)) = \phi \circ \eta_\pi$ which follows from equality \eqref{eqadj}, is an isomorphism.

Pick $a \in A$ non-zero such that $a\m_A=0$, and let $I = (a)$ be the proper ideal generated by $a$.
We have a short exact sequence of admissible smooth $A[G]$-modules
\begin{equation} \label{dev}
0 \to a\pi \to \pi \to \pi/a\pi \to 0.
\end{equation}
Note that $a\pi$ is a non-zero vector space over $k$ and $\pi/a\pi$ is a free $A/I$-module.
Since $a\pi \subset \m_A\pi$, $\phi$ factors through an $A[G]$-linear surjection $\bar{\phi} : \pi/a\pi \twoheadrightarrow \bar{\pi}$ whose kernel is $\m_A(\pi/a\pi)$, so that $(\pi/a\pi,\bar{\phi})$ is a deformation of $\bar{\pi}$ over $A/I$ which has length $\ell(A)-1$.
By our induction hypothesis, there exists a deformation $(\sigma',\phi')$ of $\bar{\sigma}$ over $A/I$ and an isomorphism of deformations of $\bar{\pi}$ over $A/I$
\begin{equation*}
\iota' : \left( \Ind_{P^-}^G \sigma',\Ind_{P^-}^G \phi' \right) \iso \left( \pi/a\pi,\bar{\phi} \right).
\end{equation*}
On the other hand, multiplication by $a$ induces an $A[G]$-linear isomorphism
\begin{equation*}
\iota : \bar{\pi} \simeq \pi/\m_A\pi \iso a\pi.
\end{equation*}

Applying the $\delta$-functor $\HOrd_P$ to the exact sequence \eqref{dev} and using the $A[L]$-linear isomorphisms $(\Ord_P \iota) \circ \epsilon_{\bar{\sigma}}$ and $(\Ord_P \iota') \circ \epsilon_{\sigma'}$ yields an exact sequence of admissible smooth $A[L]$-modules
\begin{equation} \label{HOrd}
0 \to \bar{\sigma} \to \Ord_P \pi \to \sigma' \to \HOrd[1]_P \left( \Ind_{P^-}^G \bar{\sigma} \right).
\end{equation}
Any $A[L]$-linear morphism $\sigma' \to \HOrd[1]_P (\Ind_{P^-}^G \bar{\sigma})$ must factor through $\phi'$.
Thus by Theorem \ref{thm:HOrd} and condition (b), we deduce that the last arrow of the exact sequence \eqref{HOrd} is zero, so that we can treat the rightmost term as being zero.

Now applying the exact functor $\Ind_{P^-}^G$ gives (using again equality \eqref{eqadj}) a commutative diagram of admissible smooth $A[G]$-modules
\begin{equation*}
\begin{tikzcd}
0 \rar & \Ind_{P^-}^G \bar{\sigma} \dar["\iota","\vertiso"'] \rar & \Ind_{P^-}^G \left( \Ord_P \pi \right) \dar{\eta_\pi} \rar & \Ind_{P^-}^G \sigma' \dar["\iota'","\vertiso"'] \rar & 0 \\
0 \rar & a\pi \rar & \pi \rar & \pi/a\pi \rar & 0
\end{tikzcd}
\end{equation*}
from which we deduce that $\eta_\pi$ is also an $A[G]$-linear isomorphism by the five lemma.
Moreover, we deduce that the image of $\Ind_{P^-}^G \bar{\sigma}$ in $\Ind_{P^-}^G \left( \Ord_P \pi \right)$ is $a(\Ind_{P^-}^G \left( \Ord_P \pi \right))$, thus the image of the first non-trivial arrow of the exact sequence \eqref{HOrd} is $a(\Ord_P \pi)$.

From the $A$-linear isomorphism $\eta_\pi$ and the freeness of the $A$-module $\pi$, we deduce that $\Ord_P \pi$ is a free $A$-module using the ``only if'' part of part (1) of Lemma \ref{lem:Ind}.
Furthermore, we have a commutative diagram of admissible smooth $A[L]$-modules
\begin{equation*}
\begin{tikzcd}
\Ord_P \pi \dar["\Ord_P \phi"'] \rar[two heads] & \Ord_P \left( \pi/a\pi \right) \dar["\Ord_P \bar{\phi}"'] \rar["\sim"',"\Ord_P \iota'^{-1}"] & \Ord_P \left( \Ind_{P^-}^G \sigma' \right) \dar["\Ord_P \left( \Ind_{P^-}^G \phi' \right)"'] \rar["\epsilon_{\sigma'}^{-1}","\sim"'] & \sigma' \dar[two heads]{\phi'} \\
\Ord_P \bar{\pi} \rar[equal] & \Ord_P \bar{\pi} \rar[equal] & \Ord_P \bar{\pi} \rar["\epsilon_{\bar{\sigma}}^{-1}","\sim"'] & \bar{\sigma}
\end{tikzcd}
\end{equation*}
where the composite of the upper horizontal arrows is the next to last arrow of the exact sequence \eqref{HOrd}, which is surjective with kernel $a(\Ord_P \pi)$.
Since $a(\Ord_P \pi) \subset \m_A(\Ord_P \pi)$ and $\phi'$ is surjective with kernel $\m_A \sigma'$, we deduce that $\epsilon_{\bar{\sigma}}^{-1} \circ (\Ord_P \phi)$ is surjective with kernel $\m_A (\Ord_P \pi)$.
\end{proof}

\section{Parabolic induction and deformations over Noetherian rings}

\subsection{Orthonormalizable modules}

Let $A \in \Noe(\OO)$.
An $A$-module $U$ is called $\m_A$-adically complete and separated if $U \iso \varprojlim_n U/\m_A^nU$.
If in addition each quotient $U/\m_A^nU$ is free over $A/\m_A^n$, we say that $U$ is orthonormalizable (cf. \cite[Def.~B.1]{Eme11}).

\begin{lem}\label{onb}
The following conditions are equivalent for an $A$-module $U$:
\begin{enumerate}
\item $U$ is orthonormalizable;
\item $U$ is $\m_A$-adically complete, separated, and flat;
\item There is a set $I$ such that $U$ is isomorphic to the module of decaying functions in $A^I$.
(A function $I \to A$ is called decaying if its reduction modulo $\m_A^n$ has finite support for all $n$, cf. \cite[Def.~2.1]{Yek11}.)
\item $U\simeq \varprojlim_n U_n$ where $\{U_n\}_n$ is an inverse system of $A$-modules with the following properties:
\begin{enumerate}[(i)]
\item $\m_A^n U_n=0$, and $U_n$ is free as a module over $A/\m_A^n$;
\item $U_{n+1}\otimes_{A/\m_A^{n+1}} A/\m_A^n \iso U_n$.
\end{enumerate}
\end{enumerate}
\end{lem}

\begin{proof}
The equivalence of the last three criteria is \cite[Cor.~4.5]{Yek11}.
That (1) implies (2) is part of \cite[Lem.~B.6]{Eme11}.
The converse is clear: $U/\m_A^nU$ is flat over $A/\m_A^n$ precisely when it is free.
\end{proof}

If $U$ is an $\m_A$-adically complete and separated $A$-module and $A \to A'$ is a morphism in $\Noe(\OO)$, we define the corresponding base change as the completed tensor product
\begin{equation} \label{BCnoetherian}
U \otimesh_A A' \coloneqq \varprojlim_n U \otimes_A A'/\m_{A'}^n.
\end{equation}
It is $\m_{A'}$-adically complete and separated (cf. \cite[Cor.~3.5]{Yek11}).
Orthonormalizable modules behave well with respect to base change:

\begin{lem}
Let $U$ be an orthonormalizable $A$-module and $A\to A'$ be a morphism in $\Noe(\OO)$.
Then $U \otimesh_A A'$ is an orthonormalizable $A'$-module.
\end{lem}

\begin{proof}
For any $n \geq 1$,
\begin{align*}
(U \otimesh_A A') \otimes_{A'} A'/\m_{A'}^n &\simeq (U \otimes_A A') \otimes_{A'} A'/\m_{A'}^n \\
&\simeq U \otimes_A A'/\m_{A'}^n \\
&\simeq U \otimes_A A/\m_A^n \otimes_{A/\m_A^n} A'/\m_{A'}^n \\
&\simeq U/\m_A^nU \otimes_{A/\m_A^n} A'/\m_{A'}^n.
\end{align*}
Since $U/\m_A^nU$ is free over $A/\m_A^n$ this verifies that $(U \otimesh_A A') \otimes_{A'} A'/\m_{A'}^n$ is free over $A'/\m_{A'}^n$, and hence $U \otimesh_A A'$ is orthonormalizable.
\end{proof}

\begin{rem}
When $A'$ is finite over $A$, the tensor product $U \otimes_A A'$ is already $\m_{A'}$-adically complete; in this case the above lemma is \cite[Lem.~B.6 (4)]{Eme11}.
\end{rem}

\begin{lem} \label{lem:Hom}
For any $\m_A$-adically complete and separated $A$-modules $U,U'$, there is a natural $A$-linear isomorphism
\begin{equation*}
\Hom_A \left( U', U \right) \simeq \varprojlim_n \Hom_{A/\m_A^n} \left( U' / \m_A^n U', U / \m_A^n U \right).
\end{equation*}
\end{lem}

\begin{proof}
The $A$-linear maps given by $f\mapsto (f \bmod \m_A^n)_n$ and $(f_n)\mapsto\varprojlim_n f_n$ are inverses of each other.
\end{proof}

\subsection{Continuous parabolic induction}

Let $A \in \Noe(\OO)$.
We define several categories of representations of $G$ over $A$.

An $\m_A$-adically continuous $A[G]$-module is an $\m_A$-adically complete and separated $A$-module $\pi$ endowed with an $A$-linear $G$-action such that the map $G \times \pi \to \pi$ is continuous when $\pi$ is given its $\m_A$-adic topology (equivalently, the induced action of $G$ on $\pi/\m_A^n\pi$ is smooth for all $n \geq 1$).
We let $\Mod_G^{\m_A-\cont}(A)$ be the full subcategory of $\Mod_G(A)$ consisting of $\m_A$-adically continuous $A[G]$-modules.
We let $\Mod_G^{\m_A-\cont}(A)^\fl$ be the full subcategory of $\Mod_G^{\m_A-\cont}(A)$ consisting of orthonormalizable $A[G]$-modules.

\begin{rem} \label{rem:cont}
\begin{enumerate}
\item If $A$ is Artinian, then
\begin{equation*}
\Mod_G^{\m_A-\cont}(A) = \Mod_G^\infty(A) \quad \text{and} \quad \Mod_G^{\m_A-\cont}(A)^\fl = \Mod_G^\infty(A)^\fl.
\end{equation*}
\item Emerton defined an $\m_A$-adically admissible $A[G]$-module to be an $\m_A$-adically continuous $A[G]$-module $\pi$ such that $\pi/\m_A\pi$ is admissible as a smooth $k[G]$-module (cf. \cite[Def.~3.1.13]{Eme11}).
If $A$ is Artinian, then one recovers the usual definition of admissible smooth $A[G]$-module.
\end{enumerate}
\end{rem}

\begin{defn}
For any $\m_A$-adically continuous $A[L]$-module $\sigma$, we define an $A$-module
\begin{equation*}
\Ind_{P^-}^G \sigma \coloneqq \left\{ \text{continuous $f : G \to \sigma$} \middlevert \text{$f(pg) = pf(g)$ for all $p \in P^-$ and $g \in G$} \right\}
\end{equation*}
on which we let $G$ act by right translation.
\end{defn}

\begin{rem}
If $A$ is Artinian, then one recovers the smooth parabolic induction.
\end{rem}

\begin{prop} \label{prop:Indcont}
Let $\sigma$ be an $\m_A$-adically continuous $A[L]$-module.
\begin{enumerate}
\item $\Ind_{P^-}^G \sigma$ is an $\m_A$-adically continuous $A[G]$-module, and it is orthonormalizable if and only if $\sigma$ is orthonormalizable.
\item For any morphism $A \to A'$ in $\Noe(\OO)$, there is a natural $A[G]$-linear isomorphism
\begin{equation*}
\left( \Ind_{P^-}^G \sigma \right) \otimesh_A A' \iso \Ind_{P^-}^G \left( \sigma \otimesh_A A' \right).
\end{equation*}
\end{enumerate}
\end{prop}

\begin{proof}
Let $U$ be the $A$-module underlying the representation $\sigma$.
By assumption $U$ is $\m_A$-adically complete and separated, that is $U=\varprojlim_n U/ \m_A^n U$.
Once and for all we choose a continuous section of the projection $G \twoheadrightarrow P^-\backslash G=:X$ as in the proof of Lemma \ref{lem:Ind}.
This allows us to identify $\Ind_{P^-}^G \sigma$ with $\CC(X,U)$ as an $A$-module.
We prove the Lemma in a series of steps below.

\medskip

\noindent\textbf{Step 1.}
$\CC(X,A) \otimes_A A/I \iso \CC(X,A/I)$ for any ideal $I \subset A$.

\medskip

This is done by induction on the minimal number of generators of $I$, say $n_I$.
The base case where $I=aA$ is principal amounts to the exactness of the two sequences
\begin{enumerate}[(a)]
\item $0 \longrightarrow \CC(X,A[a])\longrightarrow \CC(X,A) \overset{a}{\longrightarrow} \CC(X,aA)\longrightarrow 0$,
\item $0 \longrightarrow \CC(X,aA)\longrightarrow \CC(X,A) \longrightarrow \CC(X,A/aA)\longrightarrow 0$.
\end{enumerate}
The surjectivity in (a) follows from the existence of a continuous cross-section of the multiplication by $a$ map $A \twoheadrightarrow aA$.
Indeed it induces a homeomorphism $A/A[a] \iso aA$ (since $A$ is compact) and the projection $A \twoheadrightarrow A/A[a]$ admits a continuous section since $A$ is profinite and $A[a]$ is a closed subgroup (cf. \cite[Ch.~I, §~1, Thm.~3]{Sha72}).
For the same reason $A \twoheadrightarrow A/aA$ admits a continuous section, observing that $aA\subset A$ is closed, which shows the surjectivity in (b).

Now suppose $n_I>1$ and we know the result for ideals with fewer generators.
Once and for all we choose generators for $I$, say $I=(a_1,\ldots,a_n)$ where $n=n_I$ is minimal.
Consider the quotient ring $B=A/a_nA$ and the ideal $J=I/a_nA$.
Note that $J\subset B$ is generated by the cosets of $a_1,\ldots,a_{n-1}$, so by induction we have an isomorphism
\begin{equation*}
\CC(X,B) \otimes_B B/J \iso \CC(X,B/J).
\end{equation*}
Now $A/I\simeq B/J$ so certainly any $h \in \CC(X,A/I)$ can be lifted to $\CC(X,B)$, and in turn to $\CC(X,A)$ by the base case (or just sequence (b) with $a=a_n$), which shows that the map in Step 1 is onto.
We now argue that it is also injective.
Suppose $f \in \CC(X,A)$ takes values in $I\subset A$.
Consider the reduction $\bar{f}\in \CC(X,B)$, which then takes values in $J \subset B$.
By induction $\bar{f}=\sum_{i=1}^N \bar{c}_i \gamma_i$ with $c_i \in I$ and $\gamma_i \in \CC(X,B)$.
Again by (b) we can choose lifts, that is write $\gamma_i=\bar{f}_i$ where $f_i \in \CC(X,A)$.
Now the difference $f -\sum_{i=1}^N c_i f_i$ lies in $\CC(X,A)$ and reduces to $0$ modulo $a_nA$.
That is, it takes values in $a_nA$ and is therefore of the form $a_nf'$ for some $f' \in \CC(X,A)$ by sequence (a) above.
Altogether this shows that $f=a_nf'+\sum_{i=1}^N c_i f_i$ visibly belongs to $I \CC(X,A)$ as desired.

\medskip

\noindent\textbf{Step 2.}
We have an isomorphism $\CC(X,A) \otimesh_A U \iso\CC(X,U)$.

\medskip

The completed tensor product is defined as in \eqref{BCnoetherian}.
For any $n\geq 1$,
\begin{align*}
\CC(X,A) \otimes_A U/\m_A^nU &\simeq \CC(X,A) \otimes_A A/\m_A^n \otimes_{A/\m_A^n} U/\m_A^nU \\
&\simeq\CC(X,A/\m_A^n) \otimes_{A/\m_A^n} U/\m_A^nU \\
&\simeq \CC^\infty(X,A/\m_A^n) \otimes_{A/\m_A^n} U/\m_A^nU \\
&\simeq \CC^\infty(X,U/\m_A^nU) \\
&\simeq\CC(X,U/\m_A^nU).
\end{align*}
(In the second isomorphism we used Step 1; in the fourth we used \eqref{CC}.)
Taking the limit over $n$ yields the isomorphism.

\medskip

\noindent\textbf{Step 3.}
We have an isomorphism $\CC(X,U)/\m_A^n \CC(X,U) \iso \CC(X,U/\m_A^n U)$.

\medskip

For this step let $M \coloneqq \CC(X,A) \otimes_A U$, and let $\widehat{M}$ be its $\m_A$-adic completion.
By Step 2 (and its proof) we know that
\begin{itemize}
\item $M \otimes_A A/\m_A^n \simeq \CC(X,U/\m_A^nU)$,
\item $\widehat{M} \otimes_A A/\m_A^n \simeq \CC(X,U)/\m_A^n \CC(X,U)$.
\end{itemize}
It remains to observe that, since $A$ is Noetherian, the $\m_A$-adic completion of any $A$-module is $\m_A$-adically complete.
For finitely generated modules this is a standard application of the Artin-Rees lemma.
For infinitely generated modules such as our $M$ it is more subtle; it is the content of \cite[Cor.~3.5]{Yek11} (a result which Yekutieli attributes to Matlis, cf. Remark 3.7 in op. cit.).
We conclude that
\begin{equation*}
\widehat{M} \otimes_A A/\m_A^n \iso M \otimes_A A/\m_A^n
\end{equation*}
(cf. \cite[Thm.~1.5]{Yek11}) as desired.

\medskip

\noindent\textbf{Step 4.}
Proof of part (1) of the Proposition.

\medskip

We first observe that at least $\CC(X,U)$ is $\m_A$-adically complete and separated:
\begin{equation*}
\CC(X,U) \simeq \varprojlim_n \CC(X,U/\m_A^nU) \simeq \varprojlim_n \CC(X,U)/\m_A^n \CC(X,U)
\end{equation*}
by Step 3.
Furthermore
\begin{equation*}
\CC(X,U)/\m_A^n \CC(X,U)\simeq \CC(X,U/\m_A^nU)=\CC^\infty(X,U/\m_A^nU),
\end{equation*}
so that $\CC(X,U)/\m_A^n \CC(X,U)$ is free if and only if $U/\m_A^nU$ is free (using part (1) of Lemma \ref{lem:Ind}), hence part (1) of the Proposition.

\medskip

\noindent\textbf{Step 5.}
Proof of part (2) of the Proposition.

\medskip

For any morphism $A \to A'$ in $\Noe(\OO)$ we are to show that
\begin{equation*}
\CC(X,U) \otimesh_A A' \iso \CC(X,U \otimesh_A A').
\end{equation*}
By the same type of arguments as in Step 2 we find that for any $n \geq 1$,
\begin{align*}
\CC(X,U) \otimes_A A'/\m_{A'}^n &\simeq \CC(X,U) \otimes_A A/\m_A^n \otimes_{A/\m_A^n}A'/\m_{A'}^n \\
&\simeq \CC(X,U/\m_A^n U) \otimes_{A/\m_A^n}A'/\m_{A'}^n \\
&\simeq \CC^\infty(X,U/\m_A^n U) \otimes_{A/\m_A^n}A'/\m_{A'}^n \\
&\simeq\CC^\infty(X,U/\m_A^n U \otimes_{A/\m_A^n}A'/\m_{A'}^n ) \\
&\simeq\CC(X, (U\otimes_A A') \otimes_{A'} A'/\m_{A'}^n).
\end{align*}
(In the second isomorphism we used Step 3; in the fourth we used part (2) of Lemma \ref{lem:Ind}.)
Taking the limit over $n$ gives the result, which finishes the proof.
\end{proof}

\begin{rem} \label{rem:Indcont}
Part (2) with $A \twoheadrightarrow A/\m_A^n$ yields a natural $(A/\m_A^n)[G]$-linear isomorphism
\begin{equation*}
\left( \Ind_{P^-}^G \sigma \right) / \m_A^n \left( \Ind_{P^-}^G \sigma \right) \iso \Ind_{P^-}^G \left( \sigma / \m_A^n \sigma \right).
\end{equation*}
(This is precisely the content of Step 4 in the proof.)
This generalizes \cite[Lem.~4.1.3]{Eme10a} which treats the case $A=\OO$ where the argument simplifies significantly since $\m_A=(\varpi)$.
\end{rem}

\begin{cor}
Parabolic induction $\sigma \mapsto \Ind_{P^-}^G \sigma$ gives rise to fully faithful functors $\Mod_L^{\m_A-\cont}(A) \to \Mod_G^{\m_A-\cont}(A)$ and $\Mod_L^{\m_A-\cont}(A)^\fl \to \Mod_G^{\m_A-\cont}(A)^\fl$.
\end{cor}

\begin{proof}
Given any $\m_A$-adically continuous $A[L]$-modules $\sigma$ and $\sigma'$, we have $A$-linear isomorphisms
\begin{align*}
\Hom_{A[G]} \left( \Ind_{P^-}^G \sigma, \Ind_{P^-}^G \sigma' \right) &\cong \varprojlim_n \Hom_{A[G]} \left( \Ind_{P^-}^G \left( \sigma / \m_A \sigma \right), \Ind_{P^-}^G \left( \sigma' / \m_A \sigma' \right) \right) \\
&\cong \varprojlim_n \Hom_{A[L]} \left( \sigma / \m_A \sigma, \sigma' / \m_A \sigma' \right) \\
&\cong \Hom_{A[L]} \left( \sigma, \sigma' \right),
\end{align*}
cf. Lemma \ref{lem:Hom} (together with Remark \ref{rem:Indcont}) for the first and third ones, and \cite[Thm.~5.3]{Vig16} for the second one.
\end{proof}

\begin{rem}
We also deduce from Remark \ref{rem:Indcont} with $n=1$ that parabolic induction respects admissibility (cf. part (2) of Remark \ref{rem:cont}).
\end{rem}

\subsection{Deformation functors extended}\label{funcont}

We let $\bar{\pi}$ be a smooth $k[G]$-module.
We extend Definition \ref{def} to $A\in \Noe(\OO)$.

\begin{defn}\label{defcont}
We define several categories and functors.
\begin{enumerate}
\item A lift of $\bar{\pi}$ over $A \in \Noe(\OO)$ is a pair $(\pi,\phi)$ where
\begin{itemize}
\item $\pi$ is an object of $\Mod_G^{\m_A-\cont}(A)^\fl$,
\item $\phi: \pi \twoheadrightarrow \bar{\pi}$ is an $A[G]$-linear surjection with kernel $\m_A \pi$, i.e. which induces an $A[G]$-linear isomorphism $\pi \otimes_A k \iso \bar{\pi}$.
\end{itemize}
A morphism $\iota : (\pi,\phi) \to (\pi',\phi')$ of lifts of $\bar{\pi}$ over $A$ is an $A[G]$-linear morphism $\iota : \pi \to \pi'$ such that $\phi = \phi' \circ \iota$.
\item We define a covariant functor $\Deft_{\bar{\pi}} : \Noe(\OO) \to \Cat$ by letting $\Deft_{\bar{\pi}}(A)$ be the essentially small category of lifts of $\bar{\pi}$ over $A$ for any $A \in \Noe(\OO)$, and $\Deft_{\bar{\pi}}(\varphi) : \Deft_{\bar{\pi}}(A) \to \Deft_{\bar{\pi}}(A')$ be the base change functor for any morphism $\varphi : A \to A'$ in $\Noe(\OO)$.
\item We define the deformation functor $\Def_{\bar{\pi}} : \Noe(\OO) \to \Set$ as the composite $i \circ \Deft_{\bar{\pi}}$.
\end{enumerate}
\end{defn}

\subsection{Induced deformations extended}

We let $\bar{\sigma}$ be a smooth $k[L]$-module and we set $\bar{\pi} \coloneqq \Ind_{P^-}^G \bar{\sigma}$.
By functoriality and using part (1) of Proposition \ref{prop:Indcont}, we see that
\begin{itemize}
\item if $(\sigma,\phi)$ is a lift of $\bar{\sigma}$ over $A \in \Noe(\OO)$, then $(\Ind_{P^-}^G \sigma,\Ind_{P^-}^G \phi)$ is a lift of $\bar{\pi}$ over $A$,
\item and if $\iota : (\sigma,\phi) \to (\sigma',\phi')$ is a morphism of lifts of $\bar{\sigma}$ over $A \in \Noe(\OO)$, then $\Ind_{P^-}^G \iota : (\Ind_{P^-}^G \sigma,\Ind_{P^-}^G \phi) \to (\Ind_{P^-}^G \sigma',\Ind_{P^-}^G \phi')$ is a morphism of lifts of $\bar{\pi}$ over $A$.
\end{itemize}
Thus for any $A \in \Noe(\OO)$, we obtain a functor
\begin{equation*}
\Ind_{P^-}^G : \Deft_{\bar{\sigma}}(A) \to \Deft_{\bar{\pi}}(A)
\end{equation*}
which is functorial in $A$ by part (2) of Proposition \ref{prop:Indcont}, i.e. it is induced by a morphism $\Ind_{P^-}^G : \Deft_{\bar{\sigma}} \to \Deft_{\bar{\pi}}$ between functors $\Noe(\OO) \to \Cat$.
Thus composing the latter with $i : \Cat \to \Set$ yields a morphism
\begin{equation*}
\Ind_{P^-}^G : \Def_{\bar{\sigma}} \to \Def_{\bar{\pi}}
\end{equation*}
between functors $\Noe(\OO) \to \Set$.

\begin{lem} \label{lem:defindcont}
Let $\bar{\sigma}$ be a locally $Z_L$-finite smooth $k[L]$-module and $\bar{\pi} \coloneqq \Ind_{P^-}^G \bar{\sigma}$.
Then $\Ind_{P^-}^G : \Def_{\bar{\sigma}}(A) \to \Def_{\bar{\pi}}(A)$ is injective for any $A \in \Noe(\OO)$.
\end{lem}

\begin{proof}
Let $A \in \Noe(\OO)$.
Let $(\sigma,\phi)$ and $(\sigma',\phi')$ be two lifts of $\bar{\sigma}$ over $A$.
We denote by $(\pi,\psi)$ and $(\pi',\psi')$ the corresponding induced lifts of $\bar{\pi}$ over $A$.
Assume that there is an isomorphism $\jmath : (\pi,\psi) \iso (\pi',\psi')$ in $\Deft_{\bar{\pi}}(A)$.
Then for all $n \geq 1$, there is an induced isomorphism $\jmath_n : (\pi/\m_A^n\pi,\psi_n) \iso (\pi'/\m_A^n\pi',\psi'_n)$ in $\Deft_{\bar{\pi}}(A/\m_A^n)$ where $\psi_n : \pi/\m_A^n \pi \twoheadrightarrow \bar{\pi},\psi'_n : \pi'/\m_A^n \pi' \twoheadrightarrow \bar{\pi}$ are induced by $\psi,\psi'$ respectively, thus $\jmath_n = \Ind_{P^-}^G \imath_n$ for a unique isomorphism $\imath_n : (\sigma/\m_A^n\sigma,\phi_n) \iso (\sigma'/\m_A^n\sigma',\phi'_n)$ in $\Deft_{\bar{\sigma}}(A/\m_A^n)$ by part (1) of Lemma \ref{lem:defind} where $\phi_n : \sigma/\m_A^n \sigma \twoheadrightarrow \bar{\sigma},\phi'_n : \sigma'/\m_A^n \sigma' \twoheadrightarrow \bar{\sigma}$ are induced by $\phi,\phi'$ respectively.
Taking the projective limit over $n \geq 1$ yields an isomorphism $\imath : (\sigma,\phi) \iso (\sigma',\phi')$ in $\Deft_{\bar{\sigma}}(A)$ such that $\jmath = \Ind_{P^-}^G \imath$.
\end{proof}

\begin{thm} \label{thm:defindcont}
Let $\bar{\sigma}$ be an admissible smooth $k[L]$-module and $\bar{\pi} \coloneqq\Ind_{P^-}^G \bar{\sigma}$.
If $F=\Q_p$, then assume that
\begin{enumerate}[(a)]
\item Hypothesis \ref{hyp:HOrd} is satisfied,
\item $\Hom_L(\bar{\sigma},\bar{\sigma}^\alpha \otimes (\bar{\varepsilon}^{-1} \circ \alpha))=0$ for all $\alpha \in \Delta_L^{\perp,1}$.
\end{enumerate}
Then $\Ind_{P^-}^G : \Def_{\bar{\sigma}} \to \Def_{\bar{\pi}}$ is an isomorphism.
\end{thm}

\begin{proof}
Given Lemma \ref{lem:defindcont}, we only have to prove surjectivity.
Let $A \in \Noe(\OO)$.
Let $(\pi,\psi)$ be a lift of $\bar{\pi}$ over $A$.
For all $n \geq 1$, $(\pi/\m_A^n \pi,\psi_n)$ where $\psi_n : \pi/\m_A^n \pi \twoheadrightarrow \bar{\pi}$ is induced by $\psi$ is an object of $\Deft_{\bar{\pi}}(A / \m_A^n A)$, thus by part (1) of Theorem \ref{thm:defind} there exists a lift $(\sigma_n,\phi_n)$ of $\bar{\sigma}$ over $A/\m_A^n$ and an $A[G]$-linear isomorphism $\iota_n : \pi / \m_A^n \pi \iso \Ind_{P^-}^G \sigma_n$ such that $\psi_n = \Ind_{P^-}^G \phi_n \circ \iota_n$.
In particular, we get an $A[G]$-linear surjection $\varrho_n : \Ind_{P^-}^G \sigma_{n+1} \twoheadrightarrow \Ind_{P^-}^G \sigma_n$ with kernel $\m_A^n (\Ind_{P^-}^G \sigma_{n+1})$ and such that $\Ind_{P^-}^G \phi_{n+1} = (\Ind_{P^-}^G \phi_n) \circ \varrho_n$.
Using $\Ord_P$, we see that $\varrho_n = \Ind_{P^-}^G \rho_n$ where $\rho_n : \sigma_{n+1} \twoheadrightarrow \sigma_n$ is an $A[L]$-linear surjection with kernel $\m_A^n \sigma_{n+1}$ and such that $\phi_{n+1} = \phi_n \circ \rho_n$.
Taking projective limit over $n \geq 1$ yields a lift $(\sigma,\phi)$ of $\bar{\sigma}$ over $A$ and an $A[G]$-linear isomorphism $\iota : \pi \iso \Ind_{P^-}^G \sigma$ such that $\psi = \Ind_{P^-}^G \phi \circ \iota$.
Note that the limit $\sigma=\varprojlim_n \sigma_n$ is orthonormalizable by part (4) of Lemma \ref{onb}.
\end{proof}

\subsection{Application to Banach lifts}

We let $\Ban_G(E)$ denote the category of $E$-Banach representations of $G$ (with continuous $E[G]$-linear morphisms).
Thus its objects are Banach spaces $V$ over $E$ endowed with a jointly continuous $E$-linear action $G \times V \to V$.
We say that $V$ is unitary if its topology can be defined by a $G$-invariant norm; we write $\Ban_G(E)^\unit$ for the full subcategory consisting of unitary representations.

Following \cite[§~6.1]{Sch13} we denote by $\Ban_G(E)^{\leq 1}$ the category whose objects are $E$-Banach representations $(V,\|\cdot\|)$ of $G$ for which $\|V\|\subset |E|$ and $\|gv\|=\|v\|$ for all $g \in G$ and $v \in V$; the morphisms are the $E[G]$-linear norm-decreasing maps.
Passing to the isogeny category gives an equivalence (cf. \cite[Lem.~6.1]{Sch13} and the pertaining remark)
\begin{equation*}
\Big(\Ban_G(E)^{\leq 1}\Big)_{\Q} \iso\Ban_G(E)^\unit.
\end{equation*}

Let $(V,\|\cdot\|)$ be an object of $\Ban_G(E)^{\leq 1}$.
We use the previous definitions for $A=\OO$ with $\varpi$ instead of $\m_\OO$ in the notation.
Note that a $\varpi$-adically complete and separated $\OO$-module is orthonormalizable if and only if it is $\varpi$-torsion-free.
The unit ball $V^\circ \coloneqq \{v\in V: \|v\|\leq 1\}$ is an object of $\Mod_G^{\varpi-\cont}(\OO)^\fl$.
The functor $V \mapsto V^\circ$ yields an equivalence of categories
\begin{equation}\label{banmod}
\Ban_G(E)^{\leq 1} \iso \Mod_G^{\text{$\varpi$--cont}}(\OO)^{\text{fl}}
\end{equation}
(a quasi-inverse is given by $V^\circ \mapsto V \coloneqq V^\circ[1/p]=E \otimes_{\OO} V^\circ$ where $V$ is equipped with the gauge norm $\|v\|_{V^\circ} \coloneqq q^{-\ord_{V^\circ}(v)}$ where $\ord_{V^\circ}(v)$ is the largest integer $n$ for which $v \in \varpi^n V^\circ$).
Finally, the reduction mod $\varpi$ is an object of $\Mod_G^\infty(k)$ denoted
\begin{equation*}
\bar{V} \coloneqq V^\circ/\varpi V^\circ.
\end{equation*}
Turning the table, given an object $\bar{\pi}$ of $\Mod_G^\infty(k)$, we consider all the $E$-Banach representations $V$ of $G$ for which $\bar{V}\simeq \bar{\pi}$; the Banach lifts of $\bar{\pi}$.
Using \eqref{banmod}, we see that Banach lifts are the same as lifts over $\OO$.

\medskip

For any $E$-Banach representation $V$ of $L$, we define an $E$-Banach representation by letting $G$ act by right translation on the $E$-Banach space
\begin{equation*}
\Ind_{P^-}^G V \coloneqq \left\{ \text{continuous $f : G \to V$} \middlevert \text{$f(pg) = pf(g)$ for all $p \in P^-$ and $g \in G$} \right\}.
\end{equation*}
If $\|\cdot\|$ is an $L$-invariant norm with unit ball $V^\circ$, then the gauge norm associated to $\Ind_{P^-}^G V^\circ \subset \Ind_{P^-}^G V$ is $G$-invariant.
Thus, we obtain a functor
\begin{equation*}
\Ind_{P^-}^G : \Ban_L(E)^{\leq 1} \to \Ban_G(E)^{\leq 1}.
\end{equation*}
which corresponds to the continuous parabolic induction functor over $\OO$ under the equivalences \eqref{banmod} for $G$ and $L$.

\medskip

We conclude that Theorem \ref{thm:defindcont} with $A=\OO$ can be reformulated as follow.

\begin{cor} \label{cor:Ban}
With notation and assumptions as in Theorem \ref{thm:defindcont} above, every Banach lift of $\bar{\pi}$ is induced from a unique Banach lift of $\bar{\sigma}$ (up to isomorphism).
\end{cor}

\section{Parabolic induction and deformations over profinite rings}

\subsection{Augmented representations}

Let $A \in \Pro(\OO)$.
We generalize the definitions of \cite[§~2.1]{Eme10a} (which only considers Noetherian $A$).
Our definitions coincide with those of \cite{Sch13} (in terms of pseudocompact objects) since the residue field $k$ is finite.

A profinite (linear-topological) $A$-module $M$ is a profinite $A$-module such that there exists a fundamental system of open neighborhoods of $0$ consisting of $A$-submodules.

For example, the topological direct product $A^I$ is profinite for any set $I$ (a fundamental system is $\af^{I'}\times A^{I\backslash I'}$ for all finite subsets $I'\subset I$ and open ideals $\af\subset A$).
We say $M$ is topologically free if there is an isomorphism $M \simeq A^I$ of topological $A$-modules for some set $I$.
The subset of $M$ corresponding to the standard vectors in $A^I$ is then called a pseudobasis.

If $M$ is a profinite $A$-module and $A \to A'$ is a morphism in $\Pro(\OO)$, we define the corresponding base change, a profinite $A'$-module, as the completed tensor product
\begin{equation}\label{BCprofinite}
M \otimesh_A A' \coloneqq \varprojlim_{M',\af'} M/M' \otimes_A A'/\af'
\end{equation}
where $M' \subset M$ run over the open $A$-submodules and $\af' \subset A'$ runs over the open ideals.
For any set $I$, one has $A^{I}\otimesh_A A'=\varprojlim_{\af,I',\af'} (A/\af)^{I'} \otimes_A A'/\af'=(A')^{I}$, hence base change preserves topologically free modules.

For a compact open subgroup $K \subset G$, we let $A[[K]]$ be the completed group algebra with coefficients in $A \in \Pro(\OO)$.
That is,
\begin{equation*}
A[[K]]=\varprojlim_{K'} A[K/K']
\end{equation*}
where $K' \subset K$ runs over the open normal subgroups.
It is a profinite ring and the actual group algebra $A[K]$ sits as a dense subring in $A[[K]]$.

\begin{defn}
A profinite augmented representation of $G$ over $A$ is a profinite $A$-module $M$ with an $A$-linear $G$-action such that for some (equivalently any) compact open subgroup $K\subset G$, the induced $K$-action extends to a map $A[[K]] \times M \to M$ which is continuous.
\end{defn}

The morphisms between two profinite augmented $G$-representations are the $A[G]$-linear continuous maps (such a map is automatically $A[[K]]$-linear for any compact open subgroup $K \subset G$).
This defines a category $\Mod_G^\proaug(A)$.
We let $\Mod_G^\proaug(A)^\fl$ be the full subcategory of objects topologically free over $A$.
Given a morphism $A \to A'$ in $\Pro(\OO)$, we have $A'[[K]]=A[[K]] \otimesh_A A'$ and so $M\mapsto M \otimesh_A A'$ induces a functor $\Mod_G^\proaug(A)^\fl \to \Mod_G^\proaug(A')^\fl$.

As a preliminary observation we note that the continuity of $A[[K]] \times M \to M$ implies an a priori stronger statement (since $A$ and $M$ are profinite):

\begin{lem}\label{lineartop}
Any profinite augmented representation $M$ admits a fundamental system of open neighborhoods of $0$ consisting of $A[[K]]$-submodules.
\end{lem}

\begin{proof}
Let $M'\subset M$ be an arbitrary open $A$-submodule (thus $M/M'$ is finite).
Suppose we can find a compact open subgroup $K' \subset K$ such that $k'M'=M'$ for all $k' \in K'$.
Then
\begin{equation*}
\tilde{M} \coloneqq \bigcap_{k \in K} kM'=\bigcap_{k\in K/K'} \bigcap_{k' \in K'} kk'M'=\bigcap_{k\in K/K'} kM'
\end{equation*}
is a finite intersection, hence open, and $\tilde{M} \subset M'$ is clearly an $A[[K]]$-submodule of $M$.
It remains to show the existence of such a $K'$.
By continuity of $K \times M \to M$ at the point $(e,0)$ we can at least find a $K'' \leq K$ (compact open) and $M'' \subset M'$ (open $A$-submodule) such that $K'' \times M'' \to M'$.
Now $M'/M''$ is finite, say
\begin{equation*}
M'=\bigcup_{m \in R} (m+M''), \quad |R|<\infty.
\end{equation*}
By continuity of $K \times M \to M$ at the point $(e,m)$ we find a compact open $K_m \leq K''$ such that $K_m \times \{m\} \to m+M''$.
In particular $K_m m \subset M'$.
We may then take $K' \coloneqq \bigcap_{m \in R} K_m$ which is compact and open since $R$ is finite.
\end{proof}

\begin{rem}
The preceding observation extends in fact to arbitrary profinite rings \cite[Prop.~7.2.1]{Wil98}.
We found it useful to recall the details in our situation.
\end{rem}

\subsection{Duality over Noetherian rings}

We fix $A \in \Noe(\OO)$.
We discuss the following duality between the categories $\Mod_G^{\m_A-\cont}(A)^\fl$ and $\Mod_G^\proaug(A)^\fl$:
\begin{gather*}
\pi \mapsto \pi^\vee \coloneqq \Hom_A \left( \pi, A \right) \text{ with the topology of pointwise convergence}, \\
M \mapsto M^\vee \coloneqq \Hom_A^\cont \left( M, A \right) \text{ with the $\m_A$-adic topology}.
\end{gather*}

\begin{lem}\label{ad}
Taking $A$-linear duals defines an anti-equivalence of categories
\begin{equation*}
(-)^{\vee}: \Mod_G^{\m_A-\cont}(A)^\fl \iso \Mod_G^\proaug(A)^\fl
\end{equation*}
compatible with base change.
\end{lem}

\begin{proof}
We remove the admissibility assumption from the arguments in the proof of \cite[Prop.~3.1.12]{Eme11}. Let $K\subset G$ be a fixed compact open subgroup.
To simplify the notation, we let $A_n \coloneqq A/\m_A^nA$ for all $n \geq 1$.
According to \cite[Prop.~B.11]{Eme11} the functors in question induce an anti-equivalence between the category of orthonormalizable $A$-modules and the category of topologically free profinite $A$-modules and there is an isomorphism
\begin{equation}\label{topiso}
\Hom_A(\pi,A)/\m_A^n \Hom_A(\pi,A)\iso \Hom_{A_n}(\pi/\m_A^n\pi,A_n)
\end{equation}
for each $n \geq 1$.
Suppose $\pi$ is equipped with an $A$-linear $G$-action and fix $n$.
The induced action on $\pi/\m_A^n\pi$ is smooth if and only if the contragredient $G$-action on $\Hom_A(\pi,A)$ makes
\begin{equation*}
(\pi/\m_A^n\pi)^\vee=\Hom_{A_n}(\pi/\m_A^n\pi,A_n)
\end{equation*}
a profinite augmented $G$-representation over $A_n$.
Indeed, suppose the action on $\pi/\m_A^n\pi$ is smooth so that $\pi/\m_A^n\pi=\cup_{K'\subset K} (\pi/\m_A^n\pi)^{K'}$ where $K'\subset K$ runs over the open normal subgroups.
Then one obtains a structure of topological $A_n[[K]]=\varprojlim_{K'} A_n[K/K']$-module on
\begin{equation*}
\varprojlim_{K'} \Hom_{A_n}((\pi/\m_A^n\pi)^{K'},A_n)=(\pi/\m_A^n\pi)^\vee.
\end{equation*}
Conversely, suppose the induced $K$-action on $N \coloneqq (\pi/\m_A^n\pi)^\vee$ extends to a topological $A_n[[K]]$-module structure and let $f\in \Hom^\cont_{A_n}(N,A_n)=\pi/\m_A^n\pi$.
Since $N$ is compact and $A_n$ discrete the image of $f$ is finite and hence $\ker(f)\subset N$ is open.
By Lemma \ref{lineartop} there is an open $A_n$-submodule $N_0\subset \ker(f)$ which is $K$-stable and hence the $K$-action factors into a continuous action $N/N_0 \times K\to N/N_0$.
In particular, for any $n\in N$ we find an open normal subgroup $K_n\subset K$ such that $nh-n\in N_0$ for all $h\in K_n$.
The open normal subgroup $K'\subset K$ given as the intersection over the finitely many $K_n$ with $n\in N/N_0$ satisfies $nh-n\in N_0$ for all $n\in N, h\in K'$.
It follows
\begin{equation*}
f(nh)=f(nh-n)+f(n)=f(n)
\end{equation*}
so that $f$ is fixed by $K'$.

Next we claim that the contragredient $G$-action on $\Hom_A(\pi,A)$ makes $(\pi/\m_A^n\pi)^\vee$ a profinite augmented $G$-representation over $A_n$ for all $n$ if and only if $\Hom_A(\pi,A)$ is a profinite augmented $G$-representation over $A$.
Indeed, suppose the former holds.
The source of the isomorphism \eqref{topiso} has its quotient topology, which is profinite since $\m_A^n\Hom_A(\pi,A)$ is closed (if the ideal $\m_A^n$ is generated by $x_1,...,x_k$, the continuous map $\Hom_A(\pi,A)^{\oplus k}\to\Hom_A(\pi,A), (f_j)\mapsto \sum x_jf_j$ has image $\m_A^n\Hom_A(\pi,A)$).
Since the map \eqref{topiso} is a continuous bijection, it is a topological isomorphism.
So the contragredient action makes $\Hom_A(\pi,A)/\m_A^n \Hom_A(\pi,A)$ a topological $A_n[[K]]$-module.
Lemma \ref{lem:Hom} in connection with \eqref{topiso} shows that the natural map
\begin{equation*}
\Hom_A(\pi,A)\iso\varprojlim_n \Hom_A(\pi,A)/\m_A^n \Hom_A(\pi,A)
\end{equation*}
is a continuous bijection between profinite modules and therefore a topological isomorphism.
Now by assumption the right hand side is a topological $A[[K]]$-module (being an inverse limit of such) and this shows that $\Hom_A(\pi,A)$ is indeed a profinite augmented $G$-representation over $A$.
The converse is clear.
It follows from this discussion that $\pi$ lies in $\Mod_G^{\m_A-\cont}(A)^\fl$ if and only if $\pi^\vee$ lies in $\Mod_G^\proaug(A)^\fl$, as claimed.
Finally, given a morphism $A\to A'$ in $\Noe(\OO)$ we have
\begin{align*}
\Hom_A(\pi,A) \otimes_A A'/\m_{A'}^n &\simeq (\Hom_A(\pi,A)/\m_A^n \Hom_A(\pi,A)) \otimes_{A_n} A'/\m_{A'}^n \\
&\simeq \Hom_{A_n}(\pi/\m_A^n\pi,A_n) \otimes_{A_n} A'/\m_{A'}^n \\
&\simeq \Hom_{A'/\m_{A'}^n}(\pi/\m_A^n\pi \otimes_{A_n} A'/\m_{A'}^n,A'/\m_{A'}^n)\\
&\simeq \Hom_{A'/\m_{A'}^n}(\pi \otimes_A A'/\m_{A'}^n,A'/\m_{A'}^n).
\end{align*}
where the second equality follows from \eqref{topiso} and the third equality holds since the $A_n$-module $\pi/\m_A^n\pi$ is free.
Passing to the projective limit and using Lemma \ref{lem:Hom} yields $\Hom_A(\pi,A) \otimesh_A A'=\Hom_{A'}(\pi \otimesh_A A',A')$ as profinite $A'$-modules.
\end{proof}

\begin{rem}
\begin{enumerate}
\item If $A$ is Artinian, then Lemma \ref{ad} yields an anti-equivalence of categories
\begin{equation*}
(-)^{\vee}: \Mod_G^\infty(A)^\fl \iso \Mod_G^\proaug(A)^\fl.
\end{equation*}
If furthermore $A=\OO/\varpi^n\OO$ for some $n \geq 1$, then one recovers Pontrjagin duality (recall that $\OO/\varpi^n\OO$ is Gorenstein and hence isomorphic to the injective envelope of the $\OO/\varpi^n\OO$-module $k$).
\item If $A=\OO$, then one recovers (a $G$-equivariant version of) Schikhof duality
\begin{equation*}
(-)^{\vee}: \Mod_G^{\varpi-\cont}(\OO)^\fl \iso \Mod_G^\proaug(\OO)^\fl
\end{equation*}
(cf. \cite{Sch95}).
\item By \cite[Prop.~3.1.12]{Eme11}, $\pi$ is admissible (cf. part 2 of Remark \ref{rem:cont}) if and only if $\pi^\vee$ is finitely generated over $A[[K]]$ for some (equivalently any) compact open subgroup $K \subset G$.
\end{enumerate}
\end{rem}

We now make explicit parabolic induction through this duality.
We let $K \subset G$ be a compact open subgroup such that $G=P^- K$.
For example $K$ could be a special subgroup as in \cite[§~3.3]{Tit79}.
Observe that $P^- \twoheadrightarrow L$ is an open mapping so $P^-\cap K$ projects onto a compact open subgroup $K_L \subset L$.
In particular $A[[P^-\cap K]]$ acts on any profinite augmented representation $M$ of $L$ over $A$ via $A[[K_L]]$.
In this situation, the same definition \eqref{BCprofinite} produces a topological $A[[K]]$-module $A[[K]] \otimesh_{A[[P^-\cap K]]} M$.

\begin{prop}\label{dualind}
Let $\sigma$ be an object of $\Mod_L^{\m_A-\cont}(A)^\fl$.
There is a natural isomorphism of topological $A[[K]]$-modules
\begin{equation*}
A [[K]] \otimesh_{A[[P^- \cap K]]} \sigma^{\vee} \iso \left( \Ind_{P^-}^G \sigma \right)^\vee.
\end{equation*}
Moreover, the $G$-action on the dense $A$-submodule
\begin{equation*}
A [K] \otimes_{A[P^- \cap K]} \sigma^{\vee} \simeq A [G] \otimes_{A[P^-]} \sigma^{\vee}
\end{equation*}
coincides with the natural $G$-action on the right-hand side.
\end{prop}

\begin{proof}
As explained in the proof of Lemma \ref{lem:Ind} the projection $\text{pr}: G \twoheadrightarrow P^-\backslash G$ admits a continuous section $s: P^-\backslash G \hookrightarrow G$.
We may and will assume that its image $\Omega \coloneqq s(P^-\backslash G)$ lies in $K$ (as a compact subset).
Thus multiplication defines homeomorphisms $P^-\times \Omega \iso G$ and $(P^-\cap K) \times \Omega \iso K$.
For technical reasons which will become clear below, we compose the multiplication homeomorphism $(P^-\cap K) \times \Omega \iso K$ with inversion on $P^-\cap K$ and work with that.

We now prove the proposition in several steps.

\medskip

\noindent\textbf{Step 1.}
$\CC(K,A) \simeq \CC((P^-\cap K) \times \Omega,A) \simeq \CC(P^-\cap K,A) \otimesh_A \CC(\Omega,A).$

\medskip

The first isomorphism is clear.
By Step $2$ in the proof of Proposition \ref{prop:Indcont} any of the occuring spaces of continuous functions is $\m_A$-adically complete and separated.
Hence it suffices to prove the second isomorphism modulo $\m_A^n$.
By the same lemma we are then reduced to show
\begin{equation*}
\CC^\infty((P^-\cap K) \times \Omega,A/\m_A^n)\simeq\CC^\infty(P^-\cap K,A/\m_A^n) \otimes_{A/\m_A^n} \CC^\infty(\Omega,A/\m_A^n)
\end{equation*}
for each $n$.
However, this is clear since $\Omega$ is a compact $p$-adic manifold (and hence strictly paracompact, i.e. any open covering admits a disjoint refinement).

\medskip

\noindent\textbf{Step 2.}
For a profinite group $H$, $A[[H]]\simeq \CC(H,A)^\vee$ as topological $A[[H]]$-modules.

\medskip

If $H'\subset H$ is an open normal subgroup, the $A/\m_A^n$-module $(A/\m_A^n)[H/H']$ is free.
It follows that
\begin{equation*}
A[[H]]= \varprojlim_{n,H'} (A/\m_A^n)[H/H'] = \varprojlim_{n} (A/\m_A^n)[[H]].
\end{equation*}

On the other hand,
\begin{equation*}
(A/\m_A^n)[[H]]\simeq \CC^\infty(H,A/\m_A^n)^{\vee}= \CC(H,A/\m_A^n)^{\vee}
\end{equation*}
(cf. the argument in \cite[§.~21]{Sch11}) and
\begin{equation*}
\CC(H,A/\m_A^n)^{\vee}=\Hom_{A/\m_A^n}(\CC(H,A)/\m_A^n\CC(H,A),A/\m_A^n)
\end{equation*}
according to Step $2$ in the proof of Proposition \ref{prop:Indcont}.
Passing to the projective limit and using Lemma \ref{lem:Hom} completes Step $2$.

\medskip

\noindent\textbf{Step 3.}
$(\CC(P^-\cap K,A) \otimesh_A \CC(\Omega,A))^\vee\simeq \CC(P^-\cap K,A)^\vee \otimesh_A \CC(\Omega,A)^\vee.$

\medskip

We let $A_n \coloneqq A/\m_A^n$ and $P^-_K \coloneqq P^-\cap K$.
Since $\CC(P^-_K,A)/\m_A^n\CC(P^-_K,A)$ and $\CC(\Omega,A)/\m_A^n\CC(\Omega,A)$ are free $A_n$-modules, the $A_n$-module
\begin{equation*}
\Hom_{A_n}(\CC(P^-_K,A)/\m_A^n\CC(P^-_K,A) \otimes_{A_n} \CC(\Omega,A)/\m_A^n\CC(\Omega,A), A_n)
\end{equation*}
is canonically isomorphic to
\begin{equation*}
\Hom_{A_n}(\CC(P^-_K,A)/\m_A^n\CC(P^-_K,A),A_n) \otimes_{A_n} \Hom_{A_n}(\CC(\Omega,A)/\m_A^n\CC(\Omega,A), A_n).
\end{equation*}
In turn, the latter is isomorphic to
\begin{equation*}
\CC(P^-_K,A)^\vee/\m_A^n\CC(P^-_K,A)^\vee \otimes_{A_n} \CC(\Omega,A)^\vee/\m_A^n\CC(\Omega,A)^\vee
\end{equation*}
according to \eqref{topiso}.
Passing to the projective limit using Lemma \ref{lem:Hom} yields the claim.

\medskip

\noindent\textbf{Step 4.}
$A[[K]]\simeq A[[P^-\cap K]] \otimesh_A \CC(\Omega,A)^{\vee}$ as topological $A[[P^-\cap K]]$-modules.

\medskip

This follows from Step $1$, Step $2$ and Step $3$.

\medskip

\noindent\textbf{Step 5.}
$A[[K]] \otimesh_{A[[P^-\cap K]]} \sigma^{\vee} \simeq \CC(\Omega,\sigma)^{\vee}$ as topological $A$-modules.

\medskip

By associativity of the completed tensor product we have as topological $A$-modules
\begin{align*}
A[[K]] \otimesh_{A[[P^-\cap K]]} \sigma^{\vee} &\simeq (\CC(\Omega,A)^{\vee} \otimesh_A A[[P^-\cap K]]) \otimesh_{A[[P^-\cap K]]} \sigma^{\vee} \\
&\simeq \CC(\Omega,A)^{\vee} \otimesh_A\sigma^{\vee} \\
&\simeq (\CC(\Omega,A) \otimesh_A\sigma)^{\vee} \\
&\simeq \CC(\Omega,\sigma)^{\vee}.
\end{align*}
Here, the third equality follows as in Step $3$ and the last equality uses Step $2$ in the proof of Proposition \ref{prop:Indcont}.

\medskip

The isomorphism takes $\mu \otimes \lambda$ to the $A$-linear form $\xi \mapsto \langle \mu, \xi_{\lambda}\rangle$ on $\CC(\Omega,\sigma)$, where $\xi_{\lambda}$ is the continuous function on $K$ defined as follows:
\begin{equation*}
\xi_{\lambda}(k)=\Big((\text{pr}_{P^-\cap K}(k)^{-1}\lambda) \circ \xi \Big)(\text{pr}_{\Omega}(k)).
\end{equation*}
(note the inverse here which comes from our normalization stated before Step $1$).

\medskip

\noindent\textbf{Step 6.}
$\Ind_{P^-}^G \sigma \iso \CC(\Omega,\sigma), f \mapsto f|_{\Omega}$ as $A$-modules.

\medskip

The restriction map is injective since $G=P^-\Omega$.
Given a function $f_\Omega\in \CC(\Omega,\sigma)$ we let $f(p,x) \coloneqq pf_\Omega(x)$ for $p\in P^-,x\in\Omega$ (this is well-defined since $G\simeq P^-\times\Omega$).
Since the orbit map $p\mapsto F(p) \coloneqq pf_\Omega(.)$ lies in
\begin{equation*}
\CC(P^-,\CC(\Omega,\sigma))\simeq \CC(P^-\cap K,A) \otimesh_A \CC(\Omega,A)
\end{equation*}
the function $f$ is continuous by Step $1$ and hence a preimage of $f_\Omega$.

\medskip

\noindent\textbf{Step 7.}
Proof of the lemma.

\medskip

For a function $f$ on $G$ let $f'(g) \coloneqq f(g^{-1})$.
Consider the continuous $A$-linear map $j$ defined by the commutative diagram
\begin{equation*}
\begin{tikzcd}
A[[K]] \otimes_{A[[P^-\cap K]]} \sigma^{\vee} \dar["\iota",""'] \rar["j",""'] & (\Ind_{P^-}^G \sigma)^\vee \dar["\vertiso",""'] \\
A[[K]] \otimesh_{A[[P^-\cap K]]} \sigma^{\vee} \rar["\sim"] & \CC(\Omega,\sigma)^{\vee}
\end{tikzcd}
\end{equation*}
where $\iota$ denotes the canonical (continuous) map into the completion and the isomorphisms as topological $A$-modules come from Step $5$ and $6$.
Our definition of $\xi_{\lambda}$ implies that $j$ sends $\mu \otimes \lambda$ to the linear form $f \mapsto \langle \mu, \lambda \circ (f'|_{K})\rangle$ and so is $A[[K]]$-linear.
Hence $j$ extends to an isomorphism of topological $A[[K]]$-modules between $A[[K]] \otimesh_{A[[P^-\cap K]]} \sigma^{\vee}$ and $(\Ind_{P^-}^G \sigma)^\vee$.
The final assertion concerning the $G$-action follows from the definitions.
\end{proof}

\subsection{Deformation functors revisited}

We let $N$ be a profinite augmented representation of $G$ over $k$.
We will study the various lifts of $N$ to profinite augmented representations $M$ of $G$ over $A\in \Pro(\OO)$.
In comparison with in section \ref{funcont}, we allow non-Noetherian coefficients here.
Recall the functor $i : \Cat \to \Set$ defined by $C \mapsto \Ob(C)/\simeq$.

\begin{defn}\label{defaug}
We define several categories and functors.
\begin{enumerate}
\item A lift of $N$ over $A \in \Pro(\OO)$ is a pair $(M,\phi)$ where
\begin{itemize}
\item $M$ is an object of $\Mod_G^\proaug(A)^\fl$,
\item $\phi: M \twoheadrightarrow N$ is an $A[G]$-linear surjection with kernel $\m_A M$, i.e. which induces an $A[G]$-linear isomorphism $M \otimesh_A k \iso N$.
\end{itemize}
A morphism $\iota : (M,\phi) \to (M',\phi')$ of lifts of $N$ over $A$ is an $A[G]$-linear morphism $\iota : M \to M'$ such that $\phi = \phi' \circ \iota$.
\item We define a covariant functor $\Deft_N : \Pro(\OO) \to \Cat$ by letting $\Deft_N(A)$ be the essentially small category of lifts of $N$ over $A$ for any $A \in \Pro(\OO)$, and $\Deft_N(\varphi) : \Deft_N(A) \to \Deft_N(A')$ be the base change functor for any morphism $\varphi : A \to A'$ in $\Pro(\OO)$.
\item We define the deformation functor $\Def_N : \Pro(\OO) \to \Set$ as the composite $i \circ \Deft_N$.
\end{enumerate}
\end{defn}

The following is one of the main results from \cite{Sch13}.

\begin{thm}[T.S.] \label{rep}
Suppose $\End_{k[G]}^\cont(N)=k$.
Then:
\begin{enumerate}
\item $\Def_N$ is representable: There exists a universal deformation ring $R_N \in \Pro(\OO)$ along with bijections $\Def_N(A)\iso \Hom_{\Pro(\OO)}(R_N,A)$ functorial in $A \in \Pro(\OO)$.
\item $\tf_N \coloneqq \Def_N(k[\epsilon]) \iso \Ext_{\Mod_G^{\proaug}(k)}^1(N,N)$, where $k[\epsilon]=k[x]/(x^2)$.
\item $R_N$ is Noetherian if and only if $d_N \coloneqq \dim(\tf_N)<\infty$, in which case $R_N$ is a quotient of the formal power series ring $\OO[[x_1,\ldots,x_{d_N}]]$.
\end{enumerate}
\end{thm}

\begin{proof}
This summarizes Proposition 3.7, Theorem 3.8, and Corollary 3.9 from \cite{Sch13}.
\end{proof}

This gives a universal deformation $(M_N,\phi_N)$ in $\Deft_N(R_N)$ by choosing  a representive for the identity map in $\Def_N(R_N)\simeq \Hom_{\Pro(\OO)}(R_N,R_N)$. Any deformation $(M,\phi)$ over $A$ arises then from a unique morphism $R_N \to A$ in $\Pro(\OO)$ via base change.

\medskip

Our next result reconciles Definitions \ref{defcont} and \ref{defaug} over Noetherian rings.

\begin{lem}\label{dualdef}
Let $\bar{\pi}$ be a smooth $k[G]$-module with dual $\bar{\pi}^{\vee}=\Hom_k(\bar{\pi},k)$.
Then
\begin{enumerate}
\item $(-)^{\vee}: \widetilde{\Def}_{\bar{\pi}}(A) \to \widetilde{\Def}_{\bar{\pi}^{\vee}}(A)$ is an anti-equivalence for any $A \in \Noe(\OO)$,
\item $(-)^{\vee}: \Def_{\bar{\pi}}(A) \to \Def_{\bar{\pi}^{\vee}}(A)$ is a bijection for any $A \in \Noe(\OO)$.
\end{enumerate}
\end{lem}

\begin{proof}
Let us first make the duality functor in (1) a bit more precise.
An object $(\pi,\phi)$ of $\Deft_{\bar{\pi}}(A)$ is sent to an object $(M,\psi)$ of $\Deft_{\bar{\pi}^\vee}(A)$ by taking $M \coloneqq \pi^\vee = \Hom_A(\pi,A)$ and $\psi : \Hom_A(\pi,A) \to \Hom_k(\bar{\pi},k)$ to be the reduction mod $\m_A$.
Clearly $\psi$ is a continuous $A[G]$-linear morphism which factors through
\begin{equation*}
\psi : M \otimesh_A k=M/\overline{\m_A M} \longrightarrow \bar{\pi}^{\vee}.
\end{equation*}
To see this is an isomorphism it suffices to check that a pseudobasis is sent to a pseudobasis: Suppose $(e_i)_{i\in I}$ is a basis for $\pi$ (over $A$).
The dual $(e_i^{\vee})_{i\in I}$ is then a pseudobasis for $M=\pi^{\vee}$, and a fortiori $(e_i^{\vee}\otimes 1)_{i\in I}$ is a pseudobasis for $M \otimesh_A k$.
On the other hand $(\phi(e_i \otimes 1))_{i\in I}$ is a basis for $\bar{\pi}$.
It remains to verify that
\begin{equation*}
\psi(e_i^{\vee}\otimes 1)=\phi(e_i \otimes 1)^{\vee}
\end{equation*}
which follows straight from the definition of $\psi$.
We conclude that $(M,\psi)$ is indeed an object of $\widetilde{\Def}_{\bar{\pi}^{\vee}}(A)$.
Morphisms are dualized: If $\iota: \pi \to \pi'$ is compatible with $(\phi,\phi')$ then $\iota^{\vee}: \pi'^{\vee}\to \pi^{\vee}$ is compatible with $(\psi',\psi)$.
The same construction works in the opposite direction, and this sets up an anti-equivalence by Lemma \ref{ad}, which is furthermore compatible with base change.
Thus taking isomorphism classes yields (2).
\end{proof}

\subsection{Induced deformations revisited}

We let $\bar{\sigma}$ be a smooth $k[L]$-module with $\End_{k[L]}(\bar{\sigma})=k$ and $\bar{\pi} \coloneqq \Ind_{P^-}^G \bar{\sigma}$ (note that $\End_{k[G]}(\bar{\pi})=k$ by Theorem \ref{thm:Ord}).
By part (1) of Theorem \ref{rep}, both deformation functors $\Def_{\bar{\sigma}^{\vee}}$ and $\Def_{\bar{\pi}^{\vee}}$ are representable.
In particular, they are continuous in the sense that they commute with limits.
Since $\Pro(\OO)$ is equivalent to the category of pro-objects of $\Art(\OO)$, we deduce from Lemmas \ref{lem:defind} and \ref{dualdef} an injective morphism
\begin{equation*}
\Indd_{P^-}^G : \Def_{\bar{\sigma}^{\vee}} \hookrightarrow \Def_{\bar{\pi}^{\vee}}
\end{equation*}
between functors $\Pro(\OO) \to \Set$.
More precisely for any $A \in \Pro(\OO)$ which we write $A = \varprojlim_\af A/\af$ where $\af \subset A$ runs through the open ideals, $\Indd_{P^-}^G$ is the composite
\begin{align}\label{combine}
\Def_{\bar{\sigma}^{\vee}}(A) &\iso \varprojlim_{\af} \Def_{\bar{\sigma}^{\vee}}(A/\af) \notag \\
&\iso \varprojlim_{\af} \Def_{\bar{\sigma}}(A/\af) \notag \\
&\hookrightarrow \varprojlim_{\af} \Def_{\bar{\pi}}(A/\af) \\
&\iso \varprojlim_{\af} \Def_{\bar{\pi}^{\vee}}(A/\af) \notag \\
&\iso \Def_{\bar{\pi}^{\vee}}(A) \notag
\end{align}
where the second and the fourth isomorphisms are given by part (2) of Lemma \ref{dualdef} and the injection is given by part (2) of Lemma \ref{lem:defind}.

\medskip

The following is our main result.

\begin{thm}\label{eqgen}
Let $\bar{\sigma}$ be an admissible smooth $k[L]$-module with $\End_{k[L]}(\bar{\sigma})=k$ and $\bar{\pi} \coloneqq\Ind_{P^-}^G \bar{\sigma}$.
If $F=\Q_p$, then assume that
\begin{enumerate}[(a)]
\item Hypothesis \ref{hyp:HOrd} is satisfied,
\item $\Hom_L(\bar{\sigma},\bar{\sigma}^\alpha \otimes (\bar{\varepsilon}^{-1} \circ \alpha))=0$ for all $\alpha \in \Delta_L^{\perp,1}$.
\end{enumerate}
Then $\Indd_{P^-}^G: \Def_{\bar{\sigma}^{\vee}} \to \Def_{\bar{\pi}^{\vee}}$ is an isomorphism.
\end{thm}

\begin{proof}
This follows from part (2) of Theorem \ref{thm:defind}, which shows that \eqref{combine} is onto.
\end{proof}

Using the representability of the deformation functors, we reformulate Theorem \ref{eqgen} in terms of universal deformation rings and universal deformations.

\begin{cor}\label{univ}
With notation and assumptions as in Theorem \ref{eqgen} above, there is an isomorphism $R_{\bar{\pi}^\vee} \iso R_{\bar{\sigma}^\vee}$ in $\Pro(\OO)$ through which $M_{\bar{\pi}^\vee}=\Indd_{P^-}^G M_{\bar{\sigma}^\vee}$.
\end{cor}

Finally, under a finiteness assumption, we can express the universal deformation as the dual of a continuous parabolic induction.

\begin{cor} \label{noeth}
With notation and assumptions as in Theorem \ref{eqgen} above, if furthermore $\dim_k \Ext_L^1(\bar{\sigma},\bar{\sigma}) < \infty$, then $R_{\bar{\pi}^\vee} \simeq R_{\bar{\sigma}^\vee}$ is Noetherian and there is an $R_{\bar{\pi}^\vee}[G]$-linear isomorphism $M_{\bar{\pi}^\vee}^\vee \simeq \Ind_{P^-}^G M_{\bar{\sigma}^\vee}^\vee$.
\end{cor}

\begin{proof}
This follows from parts (2) and (3) of Theorem \ref{rep}, together with Proposition \ref{dualind}.
\end{proof}

\subsection{The case of principal series}

In this section we fix a $p$-adic torus $T$ (the $F$-points of an algebraic torus $T$ defined over $F$) and consider lifts of a given smooth character $\bar{\chi}: T \to k^\times$ over $A \in \Pro(\OO)$.
That is, continuous characters $\chi: T \to A^\times$ whose reduction mod $1+\m_A$ equals $\bar{\chi}$.
They comprise a set $\Def_{\bar{\chi}}(A)$ and the resulting functor $\Def_{\bar{\chi}} : \Pro(\OO) \to \Set$ is representable.
We give a precise description of the universal deformation.
This is standard but we find it instructive to include the details for completeness.

Let $T^{(p)} \coloneqq \varprojlim_j T/T^{p^j}$ denote the $p$-adic completion of $T$; a pro-$p$ group.
Each $T^{p^j}$ is an open subgroup, but they may not form a fundamental system of neighborhoods:
If $T=F^\times$ the intersection $\cap_{j=1}^\infty T^{p^j}=\mu_{p^\infty}'(F)$ consists of the prime-to-$p$ roots of unity in $F$.
The natural projection map $T \to T^{(p)}$ has dense image and is denoted $t \mapsto t^{(p)}$.

The Iwasawa algebra $\Lambda \coloneqq \OO[[T^{(p)}]]=\varprojlim_j \OO[T/T^{p^j}]$ is a complete local (since $T^{(p)}$ is pro-$p$, cf. \cite[§~19.7]{Sch11}) Noetherian $\OO$-algebra with residue field $k$.
Its maximal ideal $\m_{\Lambda}$ is the kernel of the reduced augmentation map $\Lambda \to \OO \to k$.
The natural map $T \to T^{(p)}\to \Lambda^\times$ is denoted by a bracket $t \mapsto [t]$; it takes values in $1+\m_{\Lambda}$ and is therefore continuous (as $T^{p^j}$ is then mapped into $1+\m_{\Lambda}^{j+1}$).

\begin{exmp}
If $T\simeq (F^\times)^n$ is a split torus $\Lambda$ can be made quite explicit.
As is well-known, by Lie theory $\OO_F^{\times}\simeq \mu_{\infty}(F) \times \Z_p^{[F:\Q_p]}$.
Therefore
\begin{equation*} \label{Lambda}
\Lambda\simeq \OO[\mu_{p^\infty}(F)^n][[X_1,\ldots,X_d]]
\end{equation*}
where $d=([F:\Q_p]+1)n$.
Note that $\mu_{p^\infty}(F)$ is trivial if $\zeta_p \notin F$.
For example if $F/\Q_p$ is unramified ($p$ odd); or just having ramification index $e<p-1$.
\end{exmp}

The universal deformation of $\bar{\chi}$ is given as follows.

\begin{prop}\label{char}
Define a character $\chi^{\univ}:T \to \Lambda^\times$ by the formula $\chi^{\univ}(t) \coloneqq \hat{\bar{\chi}}(t)[t]$ where the hat denotes the Teichmüller lift $k^\times\hookrightarrow \OO^\times$.
Then there are bijections
\begin{equation*}
\Def_{\bar{\chi}}(A) \iso \Hom_{\Pro(\OO)}(\Lambda,A),
\end{equation*}
functorial in $A \in \Pro(\OO)$; the inverse takes $\psi: \Lambda \to A$ to $\chi=\psi \circ \chi^{\univ}$.
\end{prop}

\begin{proof}
Clearly $\chi^{\univ}$ is continuous (since $t \mapsto [t]$ is) and it is a lift of $\bar{\chi}$ since $\chi^{\univ}(t)\equiv \hat{\bar{\chi}}(t) \pmod{1+\m_{\Lambda}}$.
The function $\psi \mapsto \psi \circ \chi^{\univ}$ is clearly injective since the $\psi$'s are continuous and $\OO[T] \to \Lambda$ has dense image.
Now let $\chi: T \to A^\times$ be an arbitrary lift of $\bar{\chi}$ and consider $\delta: T \to A^\times$ defined by $\delta(t) \coloneqq \chi(t)\hat{\bar{\chi}}(t)^{-1}$ which obviously takes values in $1+\m_A$ (since $\chi$ and $\hat{\bar{\chi}}$ are both lifts of $\bar{\chi}$).
Thus $\delta(T^{p^j})\subset 1 +\m_A^{j+1}$.
In particular, if $A$ is Artinian (say $\m_A^{i+1}=0$) then $\delta$ factors through $T/T^{p^i}$ which results in a map $\Lambda \to \OO[T/T^{p^i}] \to A$ with the required property.
In general write $A=\varprojlim_{\af} A/\af$ and look at the Artinian lifts $\chi_{\af}: T \to (A/\af)^\times$ obtained from $\chi$.
By what we have just observed $\chi_{\af}=\psi_{\af} \circ \chi^{\univ}$ for a unique morphism $\psi_{\af}: \Lambda \to A/\af$ in $\Pro(\OO)$.
The uniqueness guarantees they are compatible as $\af$ varies, and their limit $\psi=\varprojlim_{\af} \psi_{\af}$ is the desired morphism $\Lambda \to A$.
\end{proof}

Note that a continuous character $\chi: T \to A^\times$ is the same as an augmented $T$-representation on the $A$-module $M=A$.
Indeed, if $K \subset T$ is a compact open subgroup, a continuous character $K \to A^\times$ extends uniquely to a continuous homomorphism $A[[K]] \to A$ (cf. \cite[§~19.3]{Sch11}).
Therefore, if we take $N=\bar{\chi}^{-1}$ in Definition \ref{defaug}, the functor $\Def_N$ there coincides with our $\Def_{\bar{\chi}}$ above, and in the notation of Theorem \ref{rep} we have $R_{\bar{\chi}}\simeq \Lambda$.

Moreover, by \cite[Prop.~5.1.4]{Hau16a}, when $T\simeq (F^\times)^n$ the tangent space of $R_{\bar{\chi}}$ has dimension
\begin{equation*}
d_{\bar{\chi}}=\dim_k\Ext_T^1(\bar{\chi},\bar{\chi})=
\begin{cases}
([F:\Q_p]+1)n & \text{if $\zeta_p \notin F$} \\
([F:\Q_p]+2)n& \text{if $\zeta_p \in F$}
\end{cases}
\end{equation*}
which is perfectly coherent with the number of variables in \eqref{Lambda}, cf. part (3) of Theorem \ref{rep}.

\medskip

Now assume $G$ is quasi-split and specialize to the case where $P=B$ is a Borel subgroup.
Then the Levi factor $L=T$ is a $p$-adic torus.

\begin{cor}
Let $\bar{\chi}: T \to k^\times$ be a smooth character and $\bar{\pi} \coloneqq \Ind_{B^-}^G \bar{\chi}$.
If $F=\Q_p$, then assume that $s_\alpha(\bar{\chi}) \cdot (\bar{\varepsilon}^{-1} \circ \alpha) \neq \bar{\chi}$ for all $\alpha \in \Delta^1$.
Then $R_{\bar{\pi}^\vee} \simeq \Lambda$ is Noetherian and $M_{\bar{\pi}^\vee}^\vee \simeq \Ind_{B^-}^G \chi^\univ$.
\end{cor}

\begin{proof}
This follows immediately from Corollary \ref{noeth} applied to $\bar{\sigma}=\bar{\chi}$.
Condition (a) in Theorem \ref{eqgen} is vacuous, and condition (b) is our assumption.
By Proposition \ref{char} we know that $R_{\bar{\chi}^\vee} \simeq \Lambda$ and $M_{\bar{\chi}^\vee}$ is the lift $(\chi^{\univ})^{\vee}$ which shows the corollary.
\end{proof}

\subsection*{Acknowledgments}

C.S. would like to thank M.~Emerton for helpful correspondence in the early stages of this project, and P.~Schneider for his suggestions for improvement of the first version of \cite{Sor15}.

\end{document}